\documentclass[11pt]{article}       
\usepackage{geometry}               
\usepackage{graphicx}
\usepackage{geometry}
\usepackage{caption}
\usepackage{subcaption}
\usepackage{mathrsfs}
\usepackage{comment}
\usepackage{amsmath,esint} 
\usepackage{amssymb}  
\usepackage{amsthm}  
\usepackage{bm}
\usepackage{color}
\usepackage{array}
\usepackage{cite}
\usepackage{hyperref}
\usepackage{comment}
\usepackage{enumitem}

\numberwithin{equation}{section}
\newtheorem{theorem}{Theorem}[section]
\newtheorem{proposition}[theorem]{Proposition}
\newtheorem{defn}[theorem]{Definition}

\newtheorem{lemma}[theorem]{Lemma}
\newtheorem{remark}[theorem]{Remark}

\newtheorem*{theorem*}{Theorem}
\newtheorem*{proposition*}{Proposition}

\newcommand{\R}{\mathbb{R}}

\newcommand{\pOmega}[0]{\partial\Omega}
\newcommand{\sgrad}[0]{\nabla_{\pOmega}}
\newcommand{\closure}[1]{\overline{#1}}

\newcommand{\relaxlimitSup}[0]{{\limsup_{h\to 0}}^*}
\newcommand{\relaxlimitInf}[0]{{\liminf_{h\to 0}}_*}
\newcommand{\dL}[0]{\,\mathrm{d}\mathcal L^d}
\newcommand{\dH}[0]{\,\mathrm{d}\mathcal H^{d-1}}

\DeclareMathOperator*{\argmin}{arg\,min}

\def\le{\leqslant}
\def\leq{\leqslant}
\def\ge{\geqslant}
\def\geq{\geqslant}

\title{Convergence of a minimizing movement scheme for contact-angle mean curvature flow in a smooth bounded domain}

\newcommand{\footremember}[2]{%
    \footnote{#2}
    \newcounter{#1}
    \setcounter{#1}{\value{footnote}}%
}

\author{
  Tokuhiro Eto \footremember{trailer}{Universit\'{e} Claude Bernard Lyon 1, CNRS, Centrale Lyon, INSA Lyon, Universit\'{e} Jean Monnet, ICJ UMR5208, 69622 Villeurbanne, France (eto@math.univ-lyon1.fr),},
  Jiwoong Jang \footremember{alley}{Corresponding author. Department of Mathematics, University of Maryland-College Park (jjang124@umd.edu).
  }
}

\date{}
\providecommand{\keywords}[1]{\textbf{Key words.} #1}
\providecommand{\MSC}[1]{\textbf{MSC codes.} #1}

\begin{document}
\maketitle


\pagestyle{myheadings}
\thispagestyle{plain}

\begin{abstract}
This paper establishes an unconditional convergence result for a Chambolle-type minimizing movement scheme approximating mean curvature flow with prescribed contact angle in a smooth bounded domain. Based on the capillary functional and the geodesic signed distance relative to the container, the scheme yields a time-discrete level-set evolution. For every Lipschitz-continuous contact-angle function satisfying the natural strict nondegeneracy condition, we prove that the approximate solutions converge locally uniformly to the unique viscosity solution of the associated level-set equation with an oblique derivative boundary condition. In contrast to the previous convergence theorem, our result requires neither convexity of the container nor any curvature-type compatibility condition relating the tangential variation of the contact angle to the principal curvatures of the container boundary. The key new ingredient is a uniform Lipschitz estimate for the variational solutions, obtained through a Bernstein-type argument with additive and multiplicative correctors that rule out boundary maxima.
\end{abstract}











\vspace{0.2cm}

\MSC{35D40, 53E10, 35K93}

\section{Introduction}\label{sec:introduction}

\subsection{Brief background}

In this paper, we consider the mean curvature flow in a smooth bounded domain $\Omega\subset\R^d\,(d\ge 2)$ with general contact angle condition. Precisely, we aim to find an evolving interface $\{\Gamma(t)\}_{t\ge 0}$ satisfying the following system of equations:

\begin{equation}
\label{eq:target_equation}
    \begin{cases}
    V = - \operatorname{div}_{\Gamma}\nu_{\Gamma(t)},&\qquad \mbox{on}\quad \Gamma(t),\quad t > 0,\\
    \nu_\Omega\cdot\nu_{\Gamma(t)} = -\beta & \qquad\mbox{on}\quad \partial\Gamma(t)\cap \partial\Omega,\quad t > 0,\\
    \Gamma(0) = \Gamma_0,
    \end{cases}
\end{equation}
where $\nu_{\Gamma(t)}$ denotes the normal vector field on $\Gamma(t)$, and $V$ is the normal velocity of $\Gamma(t)$ in the direction of $\nu_{\Gamma(t)}$; the operator $\operatorname{div}_\Gamma$ denotes the surface divergence. The first equation in \eqref{eq:target_equation} means that $\Gamma(t)$ should move with the speed being equal to its mean curvature, whereas the second equation stresses that the boundary $\partial\Gamma(t)$ contacts the boundary $\partial\Omega$ of the container $\Omega$ with the angle being equal to $\arccos{(-\beta)}$. Here, $\nu_\Omega$ denotes the outward unit normal vector field on $\partial\Omega$, and the function $\beta:\partial\Omega\to(-1,1)$ is related to the prescribed contact angle function $\theta:\partial\Omega\to(0,\pi)$ by $-\beta=\cos\theta$. Here, we exclude the case where $\Gamma(t)$ is tangential to the boundary of the container. Finally, to close the system \eqref{eq:target_equation}, the third equation is imposed as the initial condition.

The first author and Giga \cite{EG24} proposed a minimizing movement scheme (MMS) to approximate the solution to \eqref{eq:target_equation}. We briefly review their idea here. To define a time-discrete solution to \eqref{eq:target_equation}, the following iteration was implemented. Given an initial level-set function $u_0$:

\begin{equation}
\label{eq:mms}
u_{n+1} := \argmin_{u\in L^2(\Omega)\cap BV(\Omega)} J_\beta(u)
\end{equation}
with
\begin{align*}
J_\beta(u) &:= C_\beta(u) + \frac{1}{2h}\int_\Omega(u - d_{\Omega,E_n})^2\,\mathrm{d}\mathcal{L}^d,\\
C_\beta(u) &:= \int_\Omega |\nabla u| + \int_{\partial\Omega}\beta\gamma u\,\mathrm{d}\mathcal{H}^{d-1},
\end{align*}
where $\mathcal{L}^d$ and $\mathcal{H}^{d-1}$ respectively denote the $d$-dimensional Lebesgue measure and $(d-1)$-dimensional Hausdorff measure; $\gamma:BV(\Omega)\to L^1(\pOmega)$ is the trace operator; $E_n := \{u_n \le 0\}$, and $d_{\Omega,E_n}$ denotes the geodesic signed distance function to $E_n$ in $\Omega$ whose sign is negative in the interior of $E_n$;
$C_\beta(u)$ is called the \textit{capillary functional}. The target energy in \eqref{eq:mms} was originally motivated by Chambolle's scheme \cite{Chambolle04} to approximate the mean curvature flow without contact boundary. In contrast to his scheme, the target energy of the MMS is replaced with the capillary functional $C_\beta(u)$ to cope with contact boundary. Moreover, it should be also noted that the data function is replaced with the geodesic signed distance function. In the case when $\Omega$ is not convex, this replacement of the distance function guarantees the monotonicity of the discrete scheme, that is, we have $d_{\Omega,E} \le d_{\Omega,F}$ whenever $E \supset F$ (see \cite[Lemma 3.5]{EG25}). We remark that these discrete schemes are the level-set analogue of Almgren--Taylor--Wang-type variational scheme to construct a flat flow \cite{Almgren93}.

Letting $\mathcal{P}(\closure{\Omega})$ denote the family of all subsets of $\closure{\Omega}$, for a capillary coefficient $\alpha:\partial\Omega\to(-1,1)$ we define a set operator $T_{h,\alpha}:\mathcal{P}(\closure{\Omega})\to \mathcal{P}(\closure{\Omega})$ by replacing $\beta$ with $\alpha$ in \eqref{eq:mms}. Namely, we set

\begin{equation}
    \label{eq:def-Th}
T_{h,\alpha}(E) := \{x\in\overline{\Omega}\mid u(x)\le 0\},
\end{equation}
where $u$ is a minimizer of \eqref{eq:mms} with $\alpha$ in place of $\beta$ and with $E\subset\closure{\Omega}$ in place of $E_n$. In fact, this set operator $T_{h,\alpha}$ is well defined since the minimizer is unique by strict convexity of the target energy and the lower semi-continuity of the capillary functional $C_\alpha(u)$ (see Modica \cite{Modica87}).
For the proof of the lower semi-continuity of $C_\alpha(u)$ which works for a uniformly $C^2$ bounded domain $\Omega$, we refer the reader to \cite[Proposition 4]{EG24}.

The geometric evolution equation \eqref{eq:target_equation} can be represented as the following Neumann boundary value problem in the level-set formulation:
\begin{align}\label{eq:level-set}
\begin{cases}
    u_t = |\nabla u| \operatorname{div} \nabla \phi(\nabla u) & \text{in} \quad \Omega \times (0,T), \\
    \nabla u \cdot \nu_{\Omega}  + \beta |\nabla u| = 0 & \text{on} \quad \partial\Omega \times (0,T), \\
    u(\cdot,0) = u_0 & \text{in} \quad \overline{\Omega},
\end{cases}
\end{align}
where $\phi(p) := |p|$ for $p\in\R^d$, and $u_0$ is supposed to be a level-set function whose zero-level set corresponds to the interface $\Gamma(t)$. We note that the first equation of \eqref{eq:level-set} is singular if $\nabla u = 0$, thus we will adopt the notion of viscosity solutions which will be defined later. The interpretation from the geometric form \eqref{eq:target_equation} to \eqref{eq:level-set} is clear on noting that if $u(\cdot, t)$ is a level-set function of $\Gamma(t)$, then $u_t(\cdot,t)/|\nabla u(\cdot,t)|$ and $\nabla u(\cdot,t) / |\nabla u(\cdot,t)|$ correspond to the normal velocity and the unit normal vector field $\nu_{\Gamma(t)}$ of $\Gamma(t)$, respectively.

We mention that the solution $w$ to the variational problem \eqref{eq:mms} solves the discrete variant of the level-set equation \eqref{eq:level-set}: The Euler--Lagrange equation of the variational problem \eqref{eq:mms} can be written as 
\begin{equation}
    \label{eq:intro-discrete-PDE}
    \begin{cases}
        w - h\operatorname{div}\nabla\phi(\nabla w) = d_{\Omega,E} & \qquad\mbox{in}\quad\Omega,\\ 
        \nabla w\cdot\nu_\Omega + \beta|\nabla w| = 0 & \qquad\mbox{on}\quad\pOmega.
    \end{cases}
\end{equation}


\subsection{The previous result on a convex domain} In \cite{EG25}, a discrete-in-time approximate solution to \eqref{eq:level-set} is constructed by using the set operator $T_h:=T_{h,-\beta}$. Precisely, given a function $u\in C(\closure{\Omega})$, the next-step function $S_hu$ of $u$ was defined by

\begin{equation}
\label{eq:def-Sh}
    S_h u(x) := \sup\left\{\lambda\in\R \mid x\in T_h(\{u \ge \lambda\})\right\}\qquad\mbox{for}\quad x\in\overline{\Omega}.
\end{equation}

The minus sign in $T_h=T_{h,-\beta}$ is forced by the superlevel-set convention in \eqref{eq:def-Sh}. Indeed, if $\alpha$ is the capillary coefficient and $E_\lambda=\{u\geq\lambda\}$ is a regular superlevel set, then
\[
    \nu_{E_\lambda}=-\frac{\nabla u}{|\nabla u|},
    \qquad
    \nu_{E_\lambda}\cdot\nu_\Omega+\alpha=0,
\]
and hence $\nabla u\cdot\nu_\Omega-\alpha|\nabla u|=0$. Choosing $\alpha=-\beta$ therefore gives exactly the boundary condition in \eqref{eq:level-set}.

The definition of $S_h$ in \eqref{eq:def-Sh} is given in \cite{EG25} (see also \cite{EGI12-1,EGI12-2} for no contact boundary case). The iteration of $S_h$ leads to the construction of approximate solution (discrete in time) to the level-set equation \eqref{eq:level-set}:

\begin{equation}
    \label{eq:approx-sol}
    u^h(t, x) := S_h^{\lfloor \frac{t}{h} \rfloor} u(x),
\end{equation}
where for $r\in\R$, $\lfloor r\rfloor$ denotes the largest integer which does not exceed $r$.
The main result of \cite{EG25} shows:

\begin{proposition}
    \label{prop:previous}
    Assume that $\Omega$ is a bounded convex domain in $\R^d$ with smooth boundary.
    Assume that $\beta\in C^{2,\alpha}(\pOmega)$ for some $\alpha\in(0,1)$, $\|\beta\|_{C^0(\pOmega)} < 1$, and \begin{align}\label{assumption:convexity}
    |\nabla_{\partial\Omega}\beta|\leq \kappa(x)\qquad\text{for all }x\in\partial\Omega,
    \end{align}
    where $\kappa(x)$ denotes the minimum positive principal curvature of $\pOmega$ at $x$.
    Then, $u^h$ uniformly converges to the unique viscosity solution to \eqref{eq:level-set} as $h\to 0$.
\end{proposition}

We explain the framework of the work \cite{EG25} of the first author and Giga, as our general framework follows this as well.
Given an initial function $u_0 \in C(\closure{\Omega})$ whose zero-level set equals $\Gamma_0$, we define $u^h(x,h) := S_hu_0(x)$. At every level $\lambda$, the superlevel set is advanced by $T_h=T_{h,-\beta}$; in particular,
\[
    \{u^h(\cdot,h)\geq\lambda\}=T_h(\{u_0\geq\lambda\}).
\]
Repeating this procedure, we have a discrete solution $u^h$ up to the time horizon $T$ as in \eqref{eq:approx-sol}.
We now consider the upper and lower relax limits $\overline{u} = {\limsup_{h\to 0}}^* u^h$ and $\underline{u} = {\liminf_{h\to 0}}_* u^h$ (see Definition~\ref{defn:relax-limit}).
If we can prove that $\overline{u}$ and $\underline{u}$ are respectively a subsolution and supersolution of the level-set equation \eqref{eq:level-set},
then a comparison principle for viscosity solutions (Proposition~\ref{prop:comparison-principle1}) together with the trivial inequality $\overline{u}\geq \underline{u}$
implies $\overline{u} = \underline{u}$ as soon as the function operator $S_h$ satisfies the monotonicity, translation invariance, stability, and consistency (Proposition~\ref{prop:BS}).
The most involved property of $S_h$ to prove is the consistency: Roughly speaking, we need to show the relation between the \textit{generator} of the function operator $S_h$ and the function $F$:

\begin{equation}
    \label{eq:intro-consistency}
    \lim_{h\to 0}\frac{S_h\varphi(x) - \varphi(x)}{h} = -F(\nabla\varphi(x), \nabla^2\varphi(x))
\end{equation}
for smooth test functions $\varphi$ with $\nabla\varphi(x)\neq 0$.
To this end, we are led to relate the level sets of $\left\{S_h\varphi\geq \lambda\right\}$ to the set $T_h(\left\{\varphi\geq \lambda\right\})$ for $\lambda\in\R$,
and if we admit this, then the consistency property will be proven as in \cite{EG25}.

This relation between $S_h$ and $T_h$ can be deduced from the continuity property of $T_h$ (Proposition~\ref{prop:relation-TS}),
and this property can be shown as a consequence of the equi-continuity of the solution $w$ to \eqref{eq:intro-discrete-PDE} with respect to the choice of the data $E$ (Proposition~\ref{prop:continuity-T}). Therefore, establishing the equi-continuity of the solution $w$ for the continuity property of $T_h$ is the most important contribution of the paper with detailed analysis, where the assumption \eqref{assumption:convexity} is used indeed critically.

In \cite{EG25} of the first author and Giga, the assumption \eqref{assumption:convexity} is used in order to use the maximum principle on the gradient function, which then implies that the solution $w$ of \eqref{eq:intro-discrete-PDE} is uniformly Lipschitz continuous. More precisely, the assumption \eqref{assumption:convexity} ensures that the gradient function is a subsolution of an elliptic equation up to the boundary. In other words,  the gradient function does not assume maximum on the boundary $\partial\Omega$ due to the assumption \eqref{assumption:convexity}. However, this argument is no longer valid if we do not assume \eqref{assumption:convexity}.


\subsection{The statement of our main result} The result of Proposition \ref{prop:previous} is restrictive in the sense that although the contact angle function $\beta$ is supposed to be prescribed, its derivative
should be bounded by geometric features of the container $\Omega$.
In particular, in the half-space case, all principal curvatures of the boundary are zero, and hence $\beta$ must be constant: This implies that in practice, we cannot prescribe spatially inhomogeneous contact angle condition. This motivates us to explore a more general criterion on $\Omega$ and $\beta$ to guarantee the convergence of the scheme. The following result completely removes the assumption \eqref{assumption:convexity}, which was used to establish the equi-continuity of the approximate solution $w$ for the continuity of $T_h$, which we now state as our main result.

\begin{theorem}\label{thm:main}
Assume that $\Omega$ is a bounded domain in $\R^d$ with smooth boundary.
Assume that $\beta$ is Lipschitz continuous on $\pOmega$ and $\|\beta\|_{C^0(\partial\Omega)}<1$.
Let $\{\beta^h\}_{h>0}$ be any sequence of smooth functions parametrized by the time step $h>0$ converging to $-\beta$ uniformly on $\pOmega$. Let $T_h$ be the set operator defined with the energy $J_{\beta^h}(u)$, and the function operator $S_h$ is defined with this $T_h$ accordingly.
Then, $u^h$ locally uniformly converges to the unique viscosity solution to \eqref{eq:level-set} as $h\to 0$.
\end{theorem}

We remark before we explain our result and its novelty that variant versions of \eqref{assumption:convexity} appeared in the literature interestingly. For instance, on a convex domain with the right-angle condition (which corresponds to the assumption \eqref{assumption:convexity} with $\beta\equiv0$), the Lipschitz continuity of viscosity solutions to \eqref{eq:level-set} is proved in \cite{GOS99}, which is directly related to the existence of stable stationary surfaces of \eqref{eq:target_equation}.
The work \cite{JKMT22} by the second author, Kwon, Mitake, and Tran includes a spatial forcing term (possibly on a nonconvex domain) and provides a sufficient condition on the forcing term and the domain to guarantee the Lipschitz continuity of viscosity solutions, which is a generalization of the condition \eqref{assumption:convexity}. The second author \cite{J23} then extends the conclusions including general angle conditions.

The main result of this paper can be seen as a continuation of this context since our paper obtains the uniform Lipschitz regularity of the solution $w$ of \eqref{eq:intro-discrete-PDE} (see Proposition \ref{prop:a-priori-estimate}) that was proved in \cite{EG25} under the assumption \eqref{assumption:convexity}. Our main result removes the assumption \eqref{assumption:convexity}, and thus, Theorem \ref{thm:main} applies on a nonconvex domain as well.

In this context, the technical novelty of our main result (see Theorem \ref{thm:main} and Proposition \ref{prop:a-priori-estimate}) is as follows: Instead of estimating the gradient function itself (denoted by $\phi$), we alternatively prove a uniform estimate for a suitable gradient $z:=\rho(\phi + \psi)$ (with an additive corrector $\psi$ and a multiplicative weight function $\rho$) in order to avoid the use of the assumption \eqref{assumption:convexity}. The weight function $\rho$ is taken so that the corrected gradient $z$ cannot attain a maximum on the boundary $\partial\Omega$, and consequently, the function $\rho$ depends on the shape of $\partial\Omega$ naturally. We next derive an elliptic partial differential inequality in $\Omega$ which is satisfied by the function $z$ and a uniform constant function. The maximum principle is thus applied so that the corrected gradient $z$ (as well as the original gradient $\phi$ up to a multiplicative constant) is bounded by a uniform constant.

\subsection{Literature}

We review previous works (by no means a complete list) on MMSs for the mean curvature flow. 
We first consider the case without contact with a container boundary.
Almgren--Taylor--Wang \cite{Almgren93} introduced a MMS for mean curvature flow in the whole space $\mathbb{R}^d$, starting from a bounded set $E_0$ of finite perimeter.  They proved the existence of a limiting flow, called a \textit{flat $\Phi$-curvature flow}, and showed that it agrees with the smooth mean curvature flow as long as the latter exists.
Luckhaus--Sturzenhecker \cite{LS95} studied the same MMS and proved convergence to a distributional solution of mean curvature flow under the additional assumption that the perimeters of the discrete solutions converge to the perimeter of the limiting flow.
Chambolle \cite{Chambolle04} proposed a level-set formulation of the Almgren--Taylor--Wang scheme.  In the present notation, the corresponding energy is given by \eqref{eq:mms} with $\beta\equiv 0$.  This formulation yields a well-defined discrete operator despite the possible non-uniqueness of set minimizers thanks to the strict convexity of the energy.
Chambolle proved that the resulting discrete flow converges, in the $L^1$ sense, to the level-set mean curvature flow up to the time horizon on which the latter is uniquely defined.
The works of Eto--Giga--Ishii \cite{EGI12-1,EGI12-2} are foundational earlier works for the approach adopted in the present paper.  They revisited Chambolle's scheme from the viewpoint of morphological operators (see \cite{C03}), and introduced a function operator associated with the discrete scheme.  By exploiting a sup-inf representation of this operator, they proved convergence of the approximate solutions to the unique viscosity solution of the corresponding level-set mean curvature equation.
Chambolle--De Gennaro--Morini \cite{CDM24} developed a MMS for anisotropic and inhomogeneous mean curvature flows with mobility and time-dependent forcing.  They proved convergence to level-set/viscosity solutions and, in low dimensions and under a suitable energy convergence assumption, convergence to distributional solutions in the sense of Luckhaus--Sturzenhecker.
The same authors \cite{CDGM26} studied a fully discrete variant of the implicit variational scheme for crystalline mean curvature flow.  Their scheme combines the time-discrete minimizing-movement approach of Almgren--Taylor--Wang and Luckhaus--Sturzenhecker with a spatial discretization based on discrete total variation energies.  A key point of their construction is a suitable discrete signed-distance/redistancing operator, which avoids the drift and pinning phenomena arising in more direct fully discrete implementations.  They proved convergence to the corresponding distributional crystalline curvature flow in arbitrary dimension for a large class of purely crystalline anisotropies.

We next recall related works in the presence of contact with a boundary.
Bellettini--Kholmatov \cite{BK18} studied a MMS for mean curvature flow of droplets in the half-space $\mathbb{R}^d_+$ with a prescribed, possibly non-constant, contact angle.  They proved the existence of generalized minimizing movements (GMMs) and showed their compatibility with a distributional formulation of mean curvature flow with prescribed contact angle.  This work motivated the subsequent works by the first author and Giga \cite{EG24,EG25}.
Kholmatov \cite{K25-2} subsequently extended this set-theoretic GMM framework to forced anisotropic curvature flows of droplets in a half-space, based on anisotropic capillary energies and prescribed anisotropic Young laws.
The same author \cite{Kholmatov25} proved consistency of GMMs with the smooth mean curvature flow of droplets in $\mathbb{R}^3_+$ with prescribed contact angle.  In particular, under suitable regularity assumptions on the initial droplet and on the contact angle function $\beta$, including $\beta\in C^{1+\alpha}(\partial\R^3_+)$ and $\|\beta\|_{C^0(\partial\R^3_+)}<1$, the GMM agrees with the corresponding smooth flow as long as the latter exists.



\subsection*{Organization of the paper}
This paper is organized as follows. In Section~\ref{sec:preliminaries},
we recall basic definitions and properties of the geodesic signed distance function, viscosity solutions, and set and function operators. 
Section~\ref{sec:lipschitz-bound} is devoted to the derivation of a Lipschitz estimate of approximate solutions. We emphasize that this section presents a main contribution of this paper.
Finally, in Section~\ref{sec:convergence}, we perform the proof of Theorem~\ref{thm:main}.

\section{Preliminaries}
\label{sec:preliminaries}
\subsection{Geodesic distance}
In this paper, we will work with the \textit{geodesic signed distance function} $d_{\Omega,E}$ to a relatively closed set $E$ in $\Omega$
instead of the ordinary signed distance function $d_E$. We give the precise definition of this function and mention some gradient bound.

\begin{defn}[Geodesic distance between points]
    For $x,y\in\closure{\Omega}$,
    a Lipschitz continuous map $\ell:[0,1]\to\closure{\Omega}$ is called a \emph{path} between $x$ and $y$, which is written as $\ell\in\operatorname{Path}(x,y)$,
    if and only if $\ell(0) = x$ and $\ell(1) = y$. The geodesic distance $\operatorname{dist}_\Omega(x,y)$ between $x$ and $y$ is defined by:

    \begin{equation}
        \label{eq:geodesic-distance}
    \operatorname{dist}_\Omega(x,y) := \inf\left\{\int_0^1|\ell'(t)|\,dt\biggm| \ell\in\operatorname{Path}(x,y)\right\}.
    \end{equation}
\end{defn}

\begin{defn}[Geodesic distance to sets]
    For each $E\subset\closure{\Omega}$, the \emph{geodesic distance function} $\operatorname{dist}_\Omega(\cdot,E)$ to the set $E$ is defined by:

    \[
    \operatorname{dist}_\Omega(x,E) := \inf\left\{\operatorname{dist}_{\Omega}(x,y)\biggm| y\in E\right\}.
    \]
\end{defn}

\begin{defn}[Geodesic signed distance to sets]
    For each $E\subset\closure{\Omega}$, the \emph{geodesic signed distance function} $d_{\Omega,E}$ to the set $E$ is defined by:

    \begin{equation*}
        d_{\Omega,E}(x) :=
        \begin{cases}
            -\operatorname{dist}_{\Omega}(x,\Omega\setminus E) & \qquad \mbox{if}\quad x\in E,\\
            \operatorname{dist}_\Omega(x,E) & \qquad \mbox{if}\quad x\notin E.
        \end{cases}
    \end{equation*}
\end{defn}

\begin{remark}
    In the case when $\Omega$ is convex, $d_{\Omega,E}$ coincides with the ordinary signed distance function $d_E$
    since the segment connecting $x$ and $y$ is always included in $\closure{\Omega}$,
    and the corresponding map $\ell\in\operatorname{Path}(x,y)$ attains the minimum of \eqref{eq:geodesic-distance}.
\end{remark}

\begin{proposition}
    \label{prop:geodesic-uniform-bounded}
    Assume that $\Omega$ is a bounded Lipschitz domain in $\R^d$. Then, there exists a constant $C_\Omega > 0$ such that
    \[
    |d_{\Omega,E}(x) - d_{\Omega,E}(y)| \leq C_\Omega|x-y|\qquad\forall x,y\in\closure{\Omega}.
    \]
    In particular, $d_{\Omega,E}$ is differentiable a.e. in $\closure{\Omega}$, and it holds that $|\nabla d_{\Omega,E}(x)|\leq C_\Omega$ for a.e. $x\in\closure{\Omega}$.
\end{proposition}
\begin{proof}
    This can be deduced from the fact that any bounded Lipschitz domain is uniform (see \cite{BH07}, for instance).
    The differentiability of $d_{\Omega,E}$ follows from Rademacher's theorem (see \cite[Theorem 3.2]{EG15}).
\end{proof}

\begin{remark}
    For the classical signed distance functions, it is well known that $|\nabla d_E| = 1$ at points of differentiability.
    For a $C^{2,1}$-domain $\Omega$, we have the bound $\|\nabla d_{\Omega,E}\|_{L^\infty(\Omega)}\leq C_\Omega$ as above.
    Note that we have $|\nabla d_{\Omega,E}(x)|\leq 1$ when differentiable since we can take $\delta > 0$ so small that
    $\operatorname{dist}_{\Omega}(x,y) = |x-y|$ for every $y\in B_\delta(x)$.
\end{remark}

\subsection{Viscosity solutions}\label{sec:viscosity-sol}
In this section, we recall the notion of viscosity solutions and its well-known properties which will be used in the sequel.
For later use, we introduce functions $F:(\R^d\setminus\{0\})\times \mathbb S^{d}\to \R$ and $B:\partial\Omega\times \R^d\to\R$,
where $\mathbb S^{d}$ denotes the set of all symmetric matrices in $\R^{d\times d}$, and $I_d\in \R^{d\times d}$ denotes the identity matrix.
Using these functions, we consider the following initial boundary problem:

\begin{equation}
\begin{aligned}
\label{eq:level-set-FB}
    u_t + F(\nabla u, \nabla^2 u) &= 0\qquad\mbox{in}\quad \Omega\times(0,T),\\
    B(\cdot, \nabla u) &= 0\qquad\mbox{on}\quad \partial\Omega\times (0,T),\\
    u(\cdot,0) &= u_0\,\,\,\quad\mbox{in}\quad\overline{\Omega},
    \end{aligned}
\end{equation}
where $\nabla^2 u$ denotes the Hessian matrix of $u$, say $(\nabla^2 u)_{ij} := \partial^2_{x_ix_j}u$ for $1\le i,j\le d$.

We note that the level-set equation \eqref{eq:level-set} can be represented as \eqref{eq:level-set-FB} by choosing
\begin{equation}
\label{eq:FB}
    \begin{aligned}
    F(p,X) &:= -\operatorname{tr}\left(\left(I_d - \frac{p\otimes p}{|p|^2}\right)X\right)\qquad\,\,\mbox{for}\quad p\in\R^d\setminus\{0\},\, X\in \mathbb S^d,\\
    B(x, p) &:= p \cdot \nu_\Omega(x) + \beta (x) |p|\qquad\quad\qquad\mbox{for}\quad x\in\overline{\Omega},\, p\in\R^d.
    \end{aligned}
\end{equation}

We now introduce the notion of viscosity solutions to \eqref{eq:level-set-FB}.
Let us begin with recalling a family of test functions from \cite[\S 2.1.3]{G06}.

\begin{defn}[Compatible test functions]
    Let $\mathcal F$ be the set of smooth functions on the half line defined by
    \[
    \mathcal F := \left\{f\in C^2[0,\infty)\,\biggm| f(0) = f'(0) = f''(0) = 0,\quad f''(r) > 0\quad\forall r > 0\right\}.
    \]
    A function $\varphi\in C^2({\Omega}\times(0,T))$ is compatible with the function $F$
    provided that for any $(\hat x, \hat t)\in\Omega\times(0,T)$ with $\nabla\varphi(\hat x, \hat t) = 0$,
    there exist a constant $\delta > 0$, a function $f \in \mathcal F$, and a modulus $\omega\in C[0,\infty)$ with $\lim_{\sigma\to 0}\omega(\sigma)/\sigma =0$
    satisfying
    \[
    |\varphi(x,t) - \varphi(\hat x, \hat t) - \varphi_t(\hat x, \hat t)(t - \hat t)| \leq f(|x- \hat x|) + \omega(|t - \hat t|)
    \]
    for every $(x,t) \in\Omega\times(0,T)$ with $|x-\hat x|\le \delta$ and $|t - \hat t|\le \delta$.
    The set of all compatible test functions with $F$ is denoted by $C^2_F(\Omega\times(0,T))$.
\end{defn}


\begin{defn}[Semi-continuous envelopes]
    The lower (resp., upper) semi-continuous envelope $F_*$ (resp., $F^*$) of $F:\R^d\setminus\{0\}\times\mathbb S^d\to\R$ is defined by
    \begin{align*}
        F_*(p,X) &:= \liminf_{\varepsilon\to 0}\left\{F(q,Y)\biggm| |p - q| < \varepsilon,\quad \|X-Y\|_2 < \varepsilon\right\}\\
        resp.,\quad F^*(p,X) &:= \limsup_{\varepsilon\to 0}\left\{F(q,Y)\biggm| |p - q| < \varepsilon,\quad \|X-Y\|_2 < \varepsilon\right\},
    \end{align*}
    where $\|X\|_2 := \sqrt{\sum_{i,j=1}^d x_{ij}^2}$ for $X = (x_{ij})_{1\leq i,j\leq d}\in \R^{d\times d}$.
\end{defn}

\begin{defn}[Viscosity solutions]
\label{dfn:viscosity-sol}
A function $u:\overline{\Omega}\times(0,T)\to \R$ is called a viscosity subsolution (resp., supersolution) to \eqref{eq:level-set-FB} provided that $u^*(x,t) < \infty$ (resp., $u_*(x,t) > -\infty$) for all $(x,t)\in\overline{\Omega}\times(0,T)$ and for any test function $\varphi\in C^2_F(\overline{\Omega}\times(0,T))$ and $(\hat x, \hat t)\in \overline{\Omega}\times(0,T)$ such that $u^*-\varphi$ takes a local maximum (resp., local minimum) at $(\hat x,\hat t)$, it holds that

\begin{equation}
\label{eq:def-viscosity-sol}
    \begin{cases}
        \varphi_t(\hat x,\hat t) + F_*(\nabla\varphi(\hat x,\hat t), \nabla^2\varphi(\hat x,\hat t)) \le 0 &\qquad\mbox{if}\quad \nabla\varphi(\hat x,\hat t) \neq 0\\
        (resp.,\, \varphi_t(\hat x,\hat t) + F^*(\nabla\varphi(\hat x,\hat t), \nabla^2\varphi(\hat x,\hat t)) \ge 0 &\qquad\mbox{if}\quad \nabla\varphi(\hat x,\hat t) \neq 0),\\
        \varphi_t(\hat x,\hat t)\le 0&\qquad \mbox{if}\quad \nabla\varphi(\hat x,\hat t) = 0\\
        (resp.,\, \varphi_t(\hat x, \hat t)\ge 0 &\qquad\mbox{if}\quad \nabla \varphi(\hat x, \hat t) = 0),
    \end{cases}
\end{equation}
if $\hat x\in \Omega$ and either \eqref{eq:def-viscosity-sol} or $B_*(\hat x, \nabla\varphi(\hat x,\hat t))\le 0$ (resp., $B^*(\hat x,\nabla\varphi(\hat x, \hat t)) \ge 0$) holds if $\hat x\in\partial\Omega$.

A function $u$ is called a viscosity solution to \eqref{eq:level-set-FB} if $u$ is a viscosity subsolution and supersolution.
\end{defn}

\begin{remark}
Definition~\ref{dfn:viscosity-sol} is based on the notion of $\mathcal F$-solutions (see e.g., \cite[Definition 2.3.7]{G06}).
For the special choice of $F$ in \eqref{eq:FB}, we have that
\begin{equation}
    \label{eq:compatible-set}
C^2_F(\Omega\times(0,T)) = \left\{\varphi\in C^2(\Omega\times(0,T))\biggm|\nabla\varphi(z) = 0\,\Longrightarrow\,\nabla^2\varphi(z) = O\right\}.
\end{equation}
Therefore, for test functions $\varphi$ with $\nabla\varphi(\hat x,\hat t) = 0$,
we may assume that $\nabla^2\varphi(\hat x, \hat t) = O$ for the confirmation of either $\varphi_t(\hat x,\hat t) \leq 0$ or $\varphi_t(\hat x,\hat t) \geq 0$.
We refer the reader to \cite[Remark 2.1.6, Proposition 2.1.8]{G06} for details.
\end{remark}

\begin{defn}[Degenerate ellipticity and geometricity]
    A function $F:(\R^d\setminus\{0\})\times\mathbb S^{d}\to\R$ is said to be degenerate elliptic if for any $p\in\R^d$ and $X,Y\in\mathbb S^d$ satisfying $X\le Y$, then it holds that
    \[F(p,X) \ge F(p,Y),
    \]
    where the condition $X\le Y$ means that $Y - X$ is positive semi-definite.
    A function $F:(\R^d\setminus\{0\})\times\mathbb S^d\to\R$ is said to be geometric if for any $\lambda > 0$, $p\in\R^d\setminus\{0\}$ and $X\in\mathbb S^d$, then it holds that
    \[F(\lambda p, \lambda X + \sigma p\otimes p) = \lambda F(p,X)\qquad\mbox{for all }\sigma\in\R.
    \]
\end{defn}

\begin{remark}
    The viscosity solution describes a weak notion of solutions in the sense that any classical solution satisfies the partial differential inequalities in \eqref{eq:def-viscosity-sol} due to the maximum principle as soon as $F$ is degenerate elliptic.
\end{remark}

\begin{remark}
    The geometric property of $F$ is important to apply the level-set method to geometric flows like the mean curvature flow \eqref{eq:target_equation}.
    This property guarantees an invariant property of the level sets yielded from the solution of \eqref{eq:level-set-FB} in regard to the choice of the initial condition $u_0$
    satisfying $\{u_0 = 0\} = \Gamma_0$. See \cite[Theorem 4.2]{B99}, for instance.
\end{remark}

We now state an important property of viscosity solutions which guarantees its uniqueness.

\noindent
\begin{proposition*}[Comparison principle]
    \label{prop:comparison-principle}
Let $u$ (resp., $v$) be a bounded viscosity subsolution (resp., supersolution) to \eqref{eq:level-set-FB}.
Assume that $F$ is degenerate elliptic and geometric.
Then, it holds that $u \le v$ in $\overline{\Omega}\times(0,T)$ provided that $u(\cdot,0)\le v(\cdot,0)$ in $\overline{\Omega}$. 
\end{proposition*}

We now recall from \cite[Theorem 2.2]{EG25} the comparison principle for \eqref{eq:level-set-FB}.
\begin{proposition}\label{prop:comparison-principle1}
    Assume that $\Omega$ is a bounded $C^{2,1}$-domain, $\|\beta\|_{C^0(\pOmega)} < 1$, and $\beta$ is Lipschitz continuous.
    Then, for the choice of $F$ and $B$ as in \eqref{eq:FB}, the target level-set equation \eqref{eq:level-set-FB} satisfies the comparison principle.
\end{proposition}

\begin{remark}[Regularity of the boundary]
    Proposition~\ref{prop:comparison-principle1} was shown in \cite[Theorem 2.1]{EG25} by invoking the comparison result due to Barles \cite[Theorem 3.1]{B99};
    the $C^{2,1}$ regularity of $\Omega$ was crucial to guarantee the uniform boundedness of the surface gradient $\sgrad\nu_\Omega$ of the outward normal vector field $\nu_\Omega$.
    We can find a similar comparison result by Ishii--Sato \cite[Theorem 2.1,\S 5]{IS04}, and their theorem also required the $C^{2,1}$ regularity of $\Omega$.
\end{remark}

\begin{remark}[Boundedness of the surface gradient of $\beta$]
    We note that $\Omega$ is not necessarily bounded in the hypothesis of Proposition~\ref{prop:comparison-principle1}.
    However, we will invoke the assumption that $\Omega$ is bounded to claim the maximum principle of the variational problem \eqref{eq:mms}.
\end{remark}

\begin{defn}[Relaxed limit]\label{defn:relax-limit}
Let $\mathcal O$ be a metric space, typically
$\overline{\Omega}$ or $\overline{\Omega}\times[0,T]$, and let
$u^h:\mathcal O\to\mathbb R$ be parametrized by $h>0$.
The upper and lower relaxed limits of $u^h$ are defined by
\[
\begin{aligned}
\overline{u}(x)
&:= \limsup_{\substack{h\to0\\y\to x}}u^h(y)
 = \lim_{\delta\downarrow0}
   \sup\left\{
   u^h(y):\,y\in\mathcal O,\ d(x,y)<\delta,\ 0<h<\delta
   \right\},\\
\underline{u}(x)
&:= \liminf_{\substack{h\to0\\y\to x}}u^h(y)
 = \lim_{\delta\downarrow0}
   \inf\left\{
   u^h(y):\,y\in\mathcal O,\ d(x,y)<\delta,\ 0<h<\delta
   \right\}.
\end{aligned}
\]
\end{defn}

In this study, we aim to apply the convergence result of discrete schemes approximating viscosity solutions by Barles and Souganidis \cite[Theorem 2.1]{BS91} to show Theorem~\ref{thm:main}. 

\begin{proposition}\label{prop:BS}
    For each $h>0$, suppose that $S_h:BUC(\overline{\Omega})\to BUC(\overline{\Omega})$ satisfies the following conditions:\newline

    \noindent
    \textbf{Monotonicity:}\newline
    \begin{equation*}
        S_h u \le S_h v\qquad\mbox{if}\quad u \le v\quad\text{in}\quad\overline{\Omega}.
    \end{equation*}

    \noindent
    \textbf{Translation invariance:}\newline
    \begin{equation*}
        \begin{aligned}
            S_h(u + c) &= S_h u + c\qquad \forall c\in\R,\\
            S_h (0) &= 0.
        \end{aligned}
    \end{equation*}

    \noindent
    \textbf{Consistency:}
    Assume that $F:(\R^d\setminus\{0\})\times\mathbb S^{d}\to\R$ is degenerate elliptic, geometric, and continuous, and $F$ satisfies $-\infty<F_*(0,O) = F^*(0,O) < \infty$.
    For every $\varphi\in C^2_F(\Omega)\cap C^2(\overline{\Omega})$ and $z\in\overline{\Omega}$ satisfying either
    \begin{itemize}
        \item $z\in\Omega$, or
        \item $z\in\partial\Omega$ and $\left<\nabla\varphi(z),\nu_\Omega(z)\right> + \beta(z)|\nabla\varphi(z)| > 0\, (resp., < 0)$,
    \end{itemize}
    it holds that
    \begin{equation}
        \label{eq:consistency}
        \begin{aligned}
        {\limsup_{h\to 0}}^*\frac{S_h\varphi(z) - \varphi(z)}{h}  &\le - F_*(\nabla\varphi(z), \nabla^2\varphi(z))\\
        ((resp.,\, {\liminf_{h\to 0}}_*\frac{S_h\varphi(z) - \varphi(z)}{h}  &\ge - F^*(\nabla\varphi(z), \nabla^2\varphi(z))).
        \end{aligned}
    \end{equation}
    Finally, assume that the limit problem \eqref{eq:level-set-FB} satisfies the comparison principle.
    Then, for any $u_0\in BUC(\overline{\Omega})$, the function $u^h$ defined by \eqref{eq:approx-sol} converges locally uniformly to the unique viscosity solution to \eqref{eq:level-set-FB}.

\end{proposition}

\begin{remark}
    We note that the original criterion of the consistency condition in \cite[Theorem 2.1]{BS91} did not include the case for $z\in\partial\Omega$.
    We do not have to check \eqref{eq:consistency} if $\left<\nabla\varphi(z),\nu_\Omega(z)\right> + \beta(z)|\nabla\varphi(z)| \le 0\, (resp.,\, \ge 0)$
    thanks to the definition of viscosity solutions in Definition~\ref{dfn:viscosity-sol}.
\end{remark}

\begin{remark}
    We mention that checking the consistency condition for compatible test functions with $F$ is sufficient, which follows by examining the proof of the convergence result by Barles--Souganidis \cite[Theorem 2.1]{BS91}.
    Thus, thanks to the relation \eqref{eq:compatible-set}, the Hessian of test functions can be supposed to be zero whenever its gradient vanishes.
\end{remark}

\subsection{Set operator and function operator}




In this section, we give a general framework which connects a set operator like $T_h$ introduced in \eqref{eq:def-Th} to a function operator via a level-set interpretation.
First, we call a map $T:\mathcal P(\closure{\Omega})\to\mathcal P(\closure{\Omega})$ a \textit{set operator} on $\closure{\Omega}$. We follow the end point convention that $T(\emptyset)=\emptyset$ and $T(\overline{\Omega})=\overline{\Omega}$.

\begin{defn}[Monotonicity of set operators]
    A set operator $T$ on $\closure{\Omega}$ is said to be monotone provided that
    $$T(E)\subset T(F)\qquad\mbox{if}\quad E\subset F\subset\closure{\Omega}.$$
\end{defn}

\begin{defn}[Continuity of set operators]
    \label{defn:Th-continuous}
    A set operator $T$ on $\closure{\Omega}$ is said to be continuous
    provided that for any non-increasing sequence $\{E_i\}_i$ of closed sets in $\closure{\Omega}$, it holds that
    \begin{equation}
        \label{eq:continuity-T}
        \bigcap_{i = 1}^\infty T(E_i) = T\left(\bigcap_{i=1}^\infty E_i\right).
    \end{equation}
\end{defn}

\begin{defn}[A function operator associated to a set operator]
    Let $T$ be a set operator on $\closure{\Omega}$. Then, the function operator $S$ associated to $T$ is defined by
    \begin{equation}
        \label{eq:defn-S}
        S u (x) := \sup\left\{\lambda\in\R\biggm| x\in T(u\geq \lambda)\right\}\qquad\mbox{for}\quad x\in\closure{\Omega},\, u\in C(\closure{\Omega}).
    \end{equation}
\end{defn}

\begin{proposition}
    \label{prop:relation-TS}
    Assume that a set operator $T$ is continuous, and let $S$ be the function operator associated to $T$.
    Then, it holds that
    \begin{equation*}
        \left\{Su \geq \lambda\right\} = T(\left\{u\geq \lambda\right\})\qquad \forall\lambda\in\R,\,\forall u\in C(\closure{\Omega}).
    \end{equation*}
\end{proposition}
\begin{proof}
    If $x\in T(\{u\geq \lambda\})$, then we deduce from the definition of $S$ \eqref{eq:defn-S} that $Su(x) \geq \lambda$, and hence $\{Su\geq \lambda\}\supset T(\{u\geq\lambda\})$.
    We show the opposite inclusion. Assume that $Su(x)\geq\lambda$.
    We can take a non-decreasing sequence $\lambda_i\uparrow Su(x)$ for which $x\in T(E_i)$ for every $i\in\mathbb N$ using the notation that $E_i := \{u\geq \lambda_i\}$.
    Since $\{E_i\}_i$ is non-increasing, we can invoke the continuity of $T$ \eqref{eq:continuity-T} and obtain
    \[
    x\in\bigcap_{i=1}^\infty T(E_i) = T\left(\bigcap_{i=1}^\infty E_i\right) = T(\{u\geq \lambda\}).
    \]
    This means $\{Su\leq \lambda\}\subset T(\{u\geq \lambda\})$ which concludes the proof.
\end{proof}

\begin{proposition}
    \label{prop:continuity-T}
    Suppose that a map $\mathcal{P}(\closure{\Omega})\ni E\mapsto w_E\in C(\closure{\Omega})$ satisfies the following conditions:\newline
    
    \noindent
    \textbf{Monotonicity:} For any $E\subset F\subset\closure{\Omega}$, $$w_E \geq w_F.$$

    \noindent
    \textbf{Uniform boundedness:} $$\sup_{E\subset\closure{\Omega}}\|w_E\|_{C^0(E)} < \infty.$$\newline

    \noindent
    \textbf{Equi-continuity:} The Lipschitz constant of $w_E$ does not depend on $E$, namely we have
    $$\sup_{E\subset\closure{\Omega}}\operatorname{Lip}(w_E) < \infty,$$
    where
    $$\operatorname{Lip}(f) := \sup_{x\neq y}\frac{|f(x) - f(y)|}{|x - y|}\qquad \mbox{for}\quad f:\closure{\Omega}\to \R.$$
    Assume that the set operator $T$ on $\closure{\Omega}$ is defined by $T(E) := \{x\in\closure{\Omega}\mid w_E(x)\leq 0\}$.
    Then, $T$ is continuous.
\end{proposition}
\begin{proof}
    Since $E\mapsto w_E$ is monotone, $T$ is obviously monotone, and therefore the left-hand side of the formula \eqref{eq:continuity-T} is larger than the right-hand side.
    The opposite inclusion can be shown by extracting a subsequence of $\left\{E_i\right\}_i$ such that $w_{E_i}$ locally uniformly converges to a function from
    the Arzerl\'a--Ascoli theorem (see \cite[Lemma 3.6]{EG25} for a similar discussion, which directly applied to the operator $T_h$ defined in \eqref{eq:def-Th}).
\end{proof}

We now revisit the capillary Chambolle-type scheme which defines the map $E\mapsto w_E^h$ (so the set operator $T_h$ is also monotone for each $h>0$).
We see that this map is monotone due to \cite[Lemma 3.1]{EG25} and is uniformly bounded from this monotonicity together with the boundedness of the domain $\Omega$.
To apply Proposition~\ref{prop:continuity-T}, it remains to show a Lipschitz bound estimate of $w_E$ which is independent of the choice of $E$.
In the strategy which was presented by the first author and Giga \cite{EG25}, they assumed the convexity of $\Omega$
and a point-wise gradient bound of the contact angle function $\beta$ by the minimal principal curvature of the boundary of 
the container $\pOmega$, as well as its smoothness (see \cite[Theorem 3.1]{EG25}).
As explained in Introduction, we only assume that $\beta$ is Lipschitz continuous and approximate it by a sequence $\{\beta^h\}_{h>0}$ of smooth functions on $\pOmega$ parametrized by the time step $h>0$. Then, we obtain a Lipschitz bound of the gradient of the approximate solution whose Lipschitz constant depends on a fixed $h>0$.

\section{Lipschitz bound estimate for the approximate solutions}
\label{sec:lipschitz-bound}
In this section, we aim to show a Lipschitz bound for the solution $w$ of the minimization problem
\begin{equation}\label{eq:minimization-with-g}
w := \argmin_{u\in L^2(\Omega)\cap BV(\Omega)} \left\{\int_\Omega |\nabla u| + \int_{\partial\Omega}\beta^h\gamma u\,\mathrm{d}\mathcal{H}^{d-1} + \frac{1}{2h}\int_\Omega(u - g)^2\,\mathrm{d}\mathcal{L}^d\right\}.
\end{equation}
Here, $\beta^h$ is a smooth function on $\partial\Omega$ that satisfies $\|\beta^h\|_{C^0(\partial\Omega)} \leq \|\beta\|_{C^0(\partial\Omega)}$ and uniformly converges to $-\beta$ on $\pOmega$ as $h\to 0$, and $g$ is a given Lipschitz continuous function on $\Omega$. As mentioned right after the statement of Theorem \ref{thm:main}, we work with the function $\beta^h$ instead of $\beta$. From a technical viewpoint, this is because we will need the classical differentiability of the function in order to implement the Bernstein method. Although the whole estimates shown in this section (see Proposition \ref{prop:a-priori-estimate}) become sensitive in $h\in(0,1)$, no issues arise when proving the continuity of the operator $T_h$ (see Lemma \ref{lem:Th_continuity}). We also stress that establishing the continuity of the operator is the main reason why we obtain gradient estimates (Proposition \ref{prop:a-priori-estimate}) by developing technical computations in this section, where all the novelties of the paper lie.

In view of \cite{EG24}, we can interpret the Euler--Lagrange equation as the Neumann boundary problem:
\begin{align}\label{eq:E-L}
\begin{cases}
w - h \operatorname{div} \nabla\phi(\nabla w)= g &\qquad \text{in}\quad \Omega, \\
\nu_{\Omega}\cdot\nabla w + \beta^h|\nabla w| = 0 &\qquad \text{on}\quad \partial \Omega.
\end{cases}
\end{align}
\noindent
However, the equation \eqref{eq:E-L} is categorized as a \textit{very singular diffusion equation} since its divergence term becomes unbounded at $|\nabla w| = 0$,
and it is involved to define its viscosity solution (see e.g.,  \cite{GGP14}).
Thus, instead of solving \eqref{eq:E-L} directly, we consider approximate energies, corresponding minimizers, and Euler-Lagrange equations as arranged below:

\begin{proposition}\label{prop:gamma-convergence}
For a smooth approximation $\{g^{\varepsilon}\}_{\varepsilon\in(0,1)}$ of $g$ satisfying $\|\nabla g^{\varepsilon}\|_{L^{\infty}(\Omega)} \leq \|\nabla g\|_{L^{\infty}(\Omega)}$, define 
\[
J_{\varepsilon,\beta^h}(u) := C_{\varepsilon,\beta^h}(u) + \frac{1}{2h}\int_\Omega(u-g^\varepsilon)^2 \,\mathrm{d}\mathcal L^d
\]
with
\[
C_{\varepsilon,\beta^h}(u) := \int_\Omega\sqrt{\varepsilon^2 + |\nabla u|^2}\dL + \int_{\pOmega}\beta^h\gamma u\,\mathrm{d}\mathcal H^{d-1}.
\]
Then, the following statements hold:
\begin{itemize}
    \item[(i)] There exists a unique minimizer $w^{\varepsilon}\in L^2(\Omega)$ of $J_{\varepsilon,\beta^h}$. Furthermore, $w^{\varepsilon}$ is smooth and solves the equation
    \begin{align}\label{eq:E-L-approximate}
    \begin{cases}
    w^{\varepsilon} - h \operatorname{div} \nabla\phi^{\varepsilon}(\nabla w^{\varepsilon}) = g^{\varepsilon} &\qquad \text{in}\quad \Omega, \\
    B^{\varepsilon}(\cdot, \nabla w^{\varepsilon}) = 0 &\qquad \text{on}\quad \partial \Omega,
    \end{cases}
    \end{align}
    where
    \begin{equation*}
    \phi^{\varepsilon}(p) := \sqrt{\varepsilon^2+|p|^2}\qquad
    \mbox{and}\qquad
    B^{\varepsilon}(x, p) := \nu_{\Omega}(x) \cdot p + \beta^h(x) \phi^{\varepsilon}(p).
    \end{equation*}

    \item[(ii)] If the minimizer $w^\varepsilon$ of the energy $J_{\varepsilon,\beta^h}$ converges to $w$ in $L^2(\Omega)$, then $w$ solves the minimization problem \eqref{eq:minimization-with-g}.
\end{itemize}
\end{proposition}

We mention that the equation \eqref{eq:E-L-approximate} is from the first variation of $J_{\varepsilon,\beta^h}$ being zero and that the smoothness of $w^{\varepsilon}$ (as stated in Proposition \ref{prop:gamma-convergence}(i)) is a consequence of a priori gradient estimates (Proposition \ref{prop:a-priori-estimate}). We also mention that \ref{prop:gamma-convergence}(ii) is a consequence of the $\Gamma$-convergence of the energy $J_{\varepsilon,\beta^h}$ to $J_{\beta^h}$ with respect to the $L^2(\Omega)$-topology as $\varepsilon\to 0$. We omit the proof of Proposition \ref{prop:gamma-convergence} as it is standard.

By proving the uniform boundedness and the equi-continuity of $w^\varepsilon$ with respect to $\varepsilon\in(0,1)$, as stated in Proposition \ref{prop:a-priori-estimate}, we will obtain the convergence of $w^{\varepsilon}$ in $C^0(\overline{\Omega})$ as $\varepsilon\to0$ (up to a subsequence). Then, Proposition \ref{prop:gamma-convergence}(ii) is applied so that we obtain the solution $w$ to the minimization problem \eqref{eq:minimization-with-g}.


\medskip

Write the condition $\|\beta\|_{C^0(\partial\Omega)}<1$ as, for some $\kappa\in\left(0,\frac12\right]$,
\begin{equation}\label{assumption:kappa}
\|\beta\|_{C^0(\partial\Omega)} \leq 1-2\kappa\quad\text{so that}\quad\|\beta^h\|_{C^0(\partial\Omega)} \leq 1-2\kappa\quad\text{for all }h\in(0,1).
\end{equation}
As the main result of this section, we have the following uniform estimate:


\begin{proposition}\label{prop:a-priori-estimate}
Let $w^{\varepsilon}$ be a smooth solution to the equation \eqref{eq:E-L-approximate}. Then, there exists a constant $C=C(\Omega,\kappa,\|\beta^h\|_{C^2(\partial\Omega)},\|g\|_{W^{1,\infty}(\Omega)})>0$ such that for any $\varepsilon,h\in(0,1)$, it holds
\begin{align*}
\|\nabla w^\varepsilon\|_{L^\infty(\Omega)} \le C.
\end{align*}
In particular, there exists a subsequence $\varepsilon\to0$, still denoted by $\varepsilon\to0$, along which $w^{\varepsilon}$ converges uniformly on $\overline{\Omega}$ to a function $w$ satisfying
\begin{align*}
\|\nabla w\|_{L^\infty(\Omega)} \le C.
\end{align*}
\end{proposition}






The remainder of this section is devoted to the proof of Proposition \ref{prop:a-priori-estimate}.

\medskip

Consider an extension $\widetilde{\beta}^h\in C^\infty(\overline{\Omega})$ of the function $\beta^h\in C^\infty(\partial\Omega)$ on the whole domain $\overline{\Omega}$ satisfying 
$$
\|\widetilde{\beta}^h\|_{C^0(\overline{\Omega})} \leq \|\beta^h\|_{C^0(\partial\Omega)}\quad\text{and}\quad\|\nabla\widetilde{\beta}^h\|_{C^1(\overline{\Omega})}\leq C(\Omega,\|\beta^h\|_{C^2(\partial\Omega)}).
$$
We use the same notation $\beta^h$ to denote the extension by abuse of notation. Let $\eta=\eta(x)$ be a smooth function on $\overline{\Omega}$ satisfying
\[
D\eta = \nu_{\Omega}\quad\text{on}\quad\partial\Omega\quad\text{and}\quad|D\eta|\leq1\quad\text{on}\quad\overline{\Omega}.
\]
We now set
\[
y^\varepsilon := \phi^\varepsilon + \beta^h D\eta \cdot Dw^{\varepsilon}, \quad \text{where } \phi^\varepsilon = \sqrt{\varepsilon^2 + |\nabla w^\varepsilon|^2}.
\]

\begin{proposition}\label{prop:w/o-multiplier}
There exists a constant $C=C(\Omega,\kappa,\|\beta^h\|_{C^2(\partial\Omega)})>0$ such that
\begin{align*}
\xi^\varepsilon(y^{\varepsilon})\leq Cy^{\varepsilon}\qquad\text{on}\quad\partial\Omega.
\end{align*}
Here, $\xi^\varepsilon:C^1(\closure{\Omega})\to\R$ is the oblique boundary operator defined by
$$
\xi^\varepsilon(f) := \nu_{\Omega}\cdot Df + \beta\frac{Df\cdot Dw^{\varepsilon}}{\sqrt{\varepsilon^2+|Dw^{\varepsilon}|^2}}\qquad\text{for}\quad f\in C^1(\overline{\Omega}).
$$

\end{proposition}
\begin{proof}
Let $x_0\in\partial\Omega$.

\medskip

\textbf{Step 1. Local coordinate setup at $x_0$.} By translation and rotation, we can assume $x_0=0$ and $\nu_{\Omega}(x_0) = -\overrightarrow{e_d}$ without loss of generality and let $(x_1,\cdots,x_{d-1})$ be the geodesic coordinate of $x_0\in\partial\Omega$. Take the transport of the tangential directions $\left\{\frac{\partial}{\partial x_i}\right\}_{i=1}^{d-1}$ along the geodesic $x_d\in[0,\delta)$ for some small number $\delta>0$ so that the geodesic coordinate is set around the point $x_0\in\partial\Omega$ (see \cite{GT83}).  Let $\nabla_{\partial\Omega}$ denote the induced connection on $\partial\Omega$ by the Euclidean connection $D$, and let the subscript $i\,(1\leq i\leq d-1)$ denote the derivative by the connection $\nabla_{\partial\Omega,i}$.

As a result of the coordinate setup, we have $D_i \eta = 0$ for $i < d$ and $D_d \eta = -1$ along the line $\{x_d\overrightarrow{e_d}\,:\,x_d\in[0,\delta)\}$. Furthermore, we can take a function $\eta\in C^{\infty}(\overline{\Omega})$ such that the second fundamental form of $\partial\Omega$ at $x_0$ is $\{D_{ij}\eta(x_0)\}_{1\leq i,j\leq d-1}$. The boundary condition $B^\epsilon = 0$ can be written as
\begin{align}\label{eq:BC-at-x0}
\frac{\partial w^\varepsilon}{\partial \nu_\Omega} + \beta^h \phi^\varepsilon = 0\quad\text{on $\partial\Omega$},\quad\text{which becomes}\quad w_d^{\varepsilon} = \beta^h \phi^\varepsilon\text{ at } x_0.
\end{align}
Also, the function $y^\varepsilon$ simplifies to
\[
y^\varepsilon = \phi^\varepsilon + \beta^h D\eta\cdot Dw^{\varepsilon} = \phi^{\varepsilon} + \beta^h (-\beta^h\phi^{\varepsilon}) = (1-(\beta^h)^2)\phi^{\varepsilon}\qquad\text{on $\partial\Omega$}.
\]

\medskip

We divide the proof into the cases when $\beta^h(x_0)=0$ and when $\beta^h(x_0)\neq0$. We start with the case $\beta^h(x_0)=0$ as this is simpler than the other case.

\medskip

\textbf{Step 2. The case when $\beta^h(x_0)=0$.} 
We compute $y_d^{\varepsilon}$ at $x_0$ to get
\[
y^\varepsilon_d = \phi_d^\varepsilon + \beta^h_d \sum_{k=1}^d \eta_k w^\varepsilon_k + \beta^h \partial_d \left(\sum_{k=1}^d \eta_k w^\varepsilon_k\right).
\]
Since $\eta_d = -1, \eta_j = 0\, (1\leq j\leq d-1)$, and $\beta^h=0$,
$w^\varepsilon_d=0$ at $x_0$, we have
\[
y^\varepsilon_d = \frac{\sum_{j=1}^{d-1}
w^\varepsilon_j w^\varepsilon_{jd} + w^\varepsilon_d
w^\varepsilon_{dd}}{\phi^\varepsilon} + \beta^h_d (-w^\varepsilon_d) + 0 =
\frac{\sum_{j=1}^{d-1} w^\varepsilon_j w^\varepsilon_{jd}}{\phi^\varepsilon}.
\]

To find $w^\varepsilon_{jd}$ when $\beta^h=0$, we differentiate the boundary condition
$w^\varepsilon \cdot \nu_\Omega + \beta^h \phi^\varepsilon = 0$ along a tangential
direction $x_j$ ($1\leq j\leq d-1$) to obtain
\[
\partial_j (w^\varepsilon \cdot \nu_\Omega) + \partial_j (\beta^h
\phi^\varepsilon) = 0.
\]
Using the Weingarten equation, we have
\[
-w^\varepsilon_{dj} + \sum_{k=1}^{d-1}
w^\varepsilon_k \eta_{jk} + \beta^h_j \phi^\varepsilon = 0 \implies
w^\varepsilon_{dj} = \sum_{k=1}^{d-1} \eta_{jk} w^\varepsilon_k + \beta^h_j
\phi^\varepsilon\quad\text{at }x_0.
\]

Substitute $w^\varepsilon_{dj}$ back into
$\xi^\varepsilon(y^\varepsilon) = -y^\varepsilon_d$ to obtain
\[
\xi^\varepsilon(y^\varepsilon) = -\frac{\sum_{j=1}^{d-1} w^\varepsilon_j
(\sum_{k=1}^{d-1} \eta_{jk} w^\varepsilon_k + \beta^h_j
\phi^\varepsilon)}{\phi^\varepsilon} = -\frac{1}{\phi^\varepsilon}
\sum_{j,k=1}^{d-1} \eta_{jk} w^\varepsilon_j w^\varepsilon_k - \sum_{j=1}^{d-1}
\beta^h_j w^\varepsilon_j.
\]
From the fact that $|w^\varepsilon_i| \le \phi^\varepsilon$ and that
\[
2\kappa\phi^{\varepsilon} \leq y^{\varepsilon} \leq (2-2\kappa)\phi^{\varepsilon},\quad\text{which follows by }|D\eta|\leq 1,
\]
we conclude $$\xi^\varepsilon(y^\varepsilon) \le Cy^\varepsilon\qquad\text{at }x_0$$ for some constant $C=C(\Omega,\kappa,\|\beta^h\|_{C^2(\partial\Omega)})>0$.

\medskip

From now on, we assume the other case when $\beta(x_0)\neq0$.

\medskip

\textbf{Step 3. Tangential derivatives $y^\varepsilon_i$ ($1 \le i \le d-1$) at $x_0$.} Differentiating $y^\varepsilon = \phi^\varepsilon + \beta^h \sum_{k=1}^d \eta_k w^\varepsilon_k$ with respect to $x_i$ for $i=1,\cdots,d-1$, we obtain
\begin{align*}
y^\varepsilon_i = D_i \phi^\varepsilon + \beta^h_i \sum_{k=1}^d \eta_k w^\varepsilon_k + \beta^h \sum_{k=1}^d (\eta_{ki} w^\varepsilon_k + \eta_k w^\varepsilon_{ki}),    
\end{align*}
which becomes
\[
\displaystyle y^\varepsilon_i = \frac{\sum_{k=1}^d w^\varepsilon_k w^\varepsilon_{ki}}{\phi^\varepsilon} - \beta^h_i w^\varepsilon_d + \beta^h \sum_{j=1}^{d-1} \eta_{ji} w^\varepsilon_j + \beta^h \eta_{di} w^\varepsilon_d - \beta^h w^\varepsilon_{di}\quad\text{at }x_0.
\]
We split $\sum_{k=1}^d w^\varepsilon_k w^\varepsilon_{ki} = \sum_{j=1}^{d-1} w^\varepsilon_j w^\varepsilon_{ji} + w^\varepsilon_d w^\varepsilon_{di}$ and substitute $w^\varepsilon_d = \beta^h \phi^\varepsilon$ to get
\[
\displaystyle y^\varepsilon_i = \frac{\sum_{j=1}^{d-1} w^\varepsilon_j w^\varepsilon_{ji}}{\phi^\varepsilon} + \frac{(\beta^h \phi^\varepsilon) w^\varepsilon_{di}}{\phi^\varepsilon} - \beta^h_i w^\varepsilon_d + \beta^h \sum_{j=1}^{d-1} \eta_{ji} w^\varepsilon_j + \beta^h \eta_{di} w^\varepsilon_d - \beta^h w^\varepsilon_{di}\quad\text{at }x_0.
\]
The cancellation $\frac{\beta^h \phi^\varepsilon w^\varepsilon_{di}}{\phi^\varepsilon} - \beta^h w^\varepsilon_{di} = 0$ happens due to \eqref{eq:BC-at-x0} so that for $1 \le i \le d-1$, we can write
\begin{equation} \label{eq:yi}
y^\varepsilon_i = \frac{\sum_{j=1}^{d-1} w^\varepsilon_j w^\varepsilon_{ji}}{\phi^\varepsilon} - \beta^h_i w^\varepsilon_d + \beta^h \sum_{j=1}^{d-1} \eta_{ji} w^\varepsilon_j + \beta^h \eta_{di} w^\varepsilon_d\quad\text{at $x_0$ for $1 \le i \le d-1$}. 
\end{equation}

\medskip

\textbf{Step 4. Mixed derivative $w^\varepsilon_{di}$ from the boundary condition.} Since $y^\varepsilon = (1-(\beta^h)^2)\phi^\varepsilon$ holds on $\partial\Omega$, we differentiate this identity tangentially to get
\[
y^\varepsilon_i = -2\beta \beta_i \phi^\varepsilon + (1-(\beta^h)^2)\frac{\sum_{j=1}^{d-1} w^\varepsilon_j w^\varepsilon_{ji} + w^\varepsilon_d w^\varepsilon_{di}}{\phi^\varepsilon}\quad\text{for $1 \le i \le d-1$ on $\partial\Omega$}.
\]
Equating with \eqref{eq:yi} and using \eqref{eq:BC-at-x0}, we obtain
\begin{align*}
\frac{\sum_{j=1}^{d-1} w^\varepsilon_j w^\varepsilon_{ji}}{\phi^\varepsilon} -& \beta^h_i (\beta^h \phi^\varepsilon) + \beta^h \sum_{j=1}^{d-1} \eta_{ji} w^\varepsilon_j + (\beta^h)^2 \eta_{di} \phi^\varepsilon = \\
&-2\beta^h \beta^h_i \phi^\varepsilon + (1-(\beta^h)^2)\frac{\sum_{j=1}^{d-1} w^\varepsilon_j w^\varepsilon_{ji}}{\phi^\varepsilon} + \beta^h(1-(\beta^h)^2) w^\varepsilon_{di}.
\end{align*}
Rearranging yields
\begin{equation} \label{eq:wdi}
(1-(\beta^h)^2) w^\varepsilon_{di} = \frac{\beta^h}{\phi^{\varepsilon}} \sum_{j=1}^{d-1} w^\varepsilon_j w^\varepsilon_{ji} + \beta^h_i \phi^\varepsilon + \sum_{j=1}^{d-1} \eta_{ji} w^\varepsilon_j + \beta^h \eta_{di} \phi^\varepsilon\quad\text{for $1 \le i \le d-1$ at $x_0$}.
\end{equation}
Here, we used the assumption that $\beta^h(x_0)\neq0$.

\medskip

\textbf{Step 5. Normal derivative $y^\varepsilon_d$.} Differentiating $y^\varepsilon$ in $x_d$, we get
\[
y^\varepsilon_d = \phi_d^\varepsilon + \beta^h_d \sum_{k=1}^d \eta_k w^\varepsilon_k + \beta^h \sum_{k=1}^d (\eta_{kd} w^\varepsilon_k + \eta_k w^\varepsilon_{kd}),
\]
which becomes
\[
y^\varepsilon_d = \frac{\sum_{j=1}^{d-1} w^\varepsilon_j w^\varepsilon_{jd} + w^\varepsilon_d w^\varepsilon_{dd}}{\phi^\varepsilon} - \beta^h_d w^\varepsilon_d + \beta^h \sum_{j=1}^{d-1} \eta_{jd} w^\varepsilon_j + \beta^h \eta_{dd} w^\varepsilon_d - \beta^h w^\varepsilon_{dd}\quad\text{ at $x_0$}.
\]
The cancellation $\frac{\beta^h \phi^\varepsilon w^\varepsilon_{dd}}{\phi^\varepsilon} - \beta^h w^\varepsilon_{dd} = 0$ happens due to \eqref{eq:BC-at-x0} so that we can write
\begin{equation} \label{eq:yd}
y^\varepsilon_d = \frac{\sum_{j=1}^{d-1} w^\varepsilon_j w^\varepsilon_{jd}}{\phi^\varepsilon} - \beta^h \beta^h_d \phi^\varepsilon + \beta^h \sum_{j=1}^{d-1} \eta_{jd} w^\varepsilon_j + (\beta^h)^2 \eta_{dd} \phi^\varepsilon\quad\text{ at $x_0$}.
\end{equation}

\medskip

\textbf{Step 6. Evaluating the boundary operator $\xi^\varepsilon(y^\varepsilon)$.} The operator $\xi^\varepsilon(y^\varepsilon)$ at $x_0$ is
\[
\xi^\varepsilon(y^\varepsilon) = \nu_\Omega \cdot Dy^\varepsilon + \beta^h \frac{Dy^\varepsilon \cdot Dw^\varepsilon}{\phi^\varepsilon} = -y^\varepsilon_d + \frac{\beta^h}{\phi^\varepsilon}\left( \sum_{i=1}^{d-1} y^\varepsilon_i w^\varepsilon_i + y^\varepsilon_d w^\varepsilon_d \right).
\]
From the boundary condition \eqref{eq:BC-at-x0}, this simplifies to
\[
\xi^\varepsilon(y^\varepsilon) = -(1-(\beta^h)^2) y^\varepsilon_d + \frac{\beta^h}{\phi^\varepsilon} \sum_{i=1}^{d-1} y^\varepsilon_i w^\varepsilon_i\quad\text{at }x_0.
\]
We expand the first term and the second term in order. By \eqref{eq:wdi} and \eqref{eq:yd}, we have
\begin{multline*}
-(1-(\beta^h)^2) y^\varepsilon_d = -\left( \frac{\beta^h}{(\phi^\varepsilon)^2} \sum_{j,k=1}^{d-1} w^\varepsilon_j w^\varepsilon_k w^\varepsilon_{kj} + \sum_{j=1}^{d-1} \beta^h_j w^\varepsilon_j + \frac{1}{\phi^\varepsilon} \sum_{j,k=1}^{d-1} \eta_{kj} w^\varepsilon_j w^\varepsilon_k + \beta^h \sum_{j=1}^{d-1} \eta_{dj} w^\varepsilon_j \right) \\
+ (1-(\beta^h)^2)\beta^h \beta^h_d \phi^\varepsilon - (1-(\beta^h)^2)\beta^h \sum_{j=1}^{d-1} \eta_{jd} w^\varepsilon_j - (1-(\beta^h)^2)(\beta^h)^2 \eta_{dd} \phi^\varepsilon.
\end{multline*}
On the other hand, by \eqref{eq:yi}, multiplying by $w^\varepsilon_i$ and summing, the second term becomes
\[
\frac{\beta^h}{\phi^\varepsilon} \sum_{i=1}^{d-1} y^\varepsilon_i w^\varepsilon_i = \frac{\beta^h}{(\phi^\varepsilon)^2} \sum_{i,j=1}^{d-1} w^\varepsilon_i w^\varepsilon_j w^\varepsilon_{ji} - (\beta^h)^2 \sum_{i=1}^{d-1} \beta^h_i w^\varepsilon_i + \frac{(\beta^h)^2}{\phi^\varepsilon} \sum_{i,j=1}^{d-1} \eta_{ji} w^\varepsilon_i w^\varepsilon_j + (\beta^h)^3 \sum_{i=1}^{d-1} \eta_{di} w^\varepsilon_i.
\]
As a consequence of adding the terms, the $\sum w^\varepsilon_i w^\varepsilon_j w^\varepsilon_{ji}$ terms cancel. Therefore, we obtain
\[
\xi^\varepsilon(y^\varepsilon) = (1-(\beta^h)^2) \left( -\frac{1}{\phi^\varepsilon} \sum_{i,j=1}^{d-1} \eta_{ij} w^\varepsilon_i w^\varepsilon_j -\frac{1+(\beta^h)^2}{1-(\beta^h)^2} \sum_{i=1}^{d-1} \beta^h_i w^\varepsilon_i - 2\beta^h \sum_{i=1}^{d-1} \eta_{di} w^\varepsilon_i + \beta^h(\beta^h_d - \beta^h \eta_{dd})\phi^\varepsilon \right).
\]

Now, as in the last part of Step 2, we finally conclude that $$\xi^\varepsilon(y^\varepsilon) \le Cy^\varepsilon\qquad\text{at }x_0$$ for some constant $C=C(\Omega,\kappa,\|\beta^h\|_{C^2(\partial\Omega)})>0$. As $x_0\in\partial\Omega$ was arbitrary, we obtain the conclusion in all cases and finish the proof.
\end{proof}

\medskip

We consider the function $z^{\varepsilon}:=\rho y^{\varepsilon}$, where $\rho\in C^{\infty}(\overline{\Omega})$ is a multiplier that will be chosen as follows.

\begin{proposition}\label{prop:multiplier}
There exists a function $\rho\in C^{\infty}(\overline{\Omega})$ that depends only on $\Omega,\kappa,\|\beta^h\|_{C^2(\partial\Omega)}$ such that $\rho\geq1$ on $\overline{\Omega}$, $\rho=1$ on $\partial\Omega$, and $z^{\varepsilon}=\rho y^{\varepsilon}$ satisfies
\begin{equation*}
\xi^\varepsilon(z^{\varepsilon})<0\qquad\text{on}\quad\pOmega.
\end{equation*}
\end{proposition}
\begin{proof}
Consider a function $\rho\in C^{\infty}(\overline{\Omega})$ that is of the form
\begin{align*}
\rho = 1 + Kd_{\partial\Omega}\qquad\text{near the boundary }\partial\Omega
\end{align*}
and $\rho\geq1$ on $\overline{\Omega}$, $\rho=1$ on $\partial\Omega$. Here, $d_{\partial\Omega}=\mathrm{dist}(\cdot,\partial\Omega)$ and $K=K(\Omega,\kappa,\|\beta^h\|_{C^2(\partial\Omega)})>0$ is a constant to be determined. Then, it holds that
\begin{align*}
\xi^\varepsilon(\rho) = \left(\nu_{\Omega}+\beta\frac{\nabla w^{\varepsilon}}{\sqrt{\varepsilon^2+|\nabla w^{\varepsilon}|^2}}\right) \cdot \left(-K\nu_{\Omega}\right) \leq -K + K(1-2\kappa) = -2\kappa K\qquad\text{on}\quad\partial\Omega.
\end{align*}
Therefore, by Proposition \ref{prop:w/o-multiplier}, since $y^\varepsilon > 0$ for every $\varepsilon > 0$, we have
\begin{align*}
\xi^\varepsilon(z^{\varepsilon}) = \rho \xi^\varepsilon(y^{\varepsilon}) + y^{\varepsilon} \xi^\varepsilon(\rho)  \leq Cy^{\varepsilon} - 2\kappa K y^{\varepsilon} = y^{\varepsilon}(C-2\kappa K)<0\qquad\text{on}\quad\partial\Omega
\end{align*}
once we choose $K=K(\Omega,\kappa,\|\beta^h\|_{C^2(\partial\Omega)})>0$ such that $K>\frac{C}{2\kappa}$. Here, $C(\Omega,\kappa,\|\beta^h\|_{C^2(\partial\Omega)})>0$ is the constant from Proposition \ref{prop:w/o-multiplier}. We finish the proof.
\end{proof}

\begin{proposition}\label{prop:equation-v}
Let the subscript $k=1,\cdots,d$ denote the partial derivative with respect to the variable $x_k$. Let $L$ denote the operator given by
\[
L[u] = u - h \phi^{\varepsilon}_{ij}(\nabla w^{\varepsilon}) \partial_{ij} u.
\]

Then, there exists a constant $C=C(\Omega,\kappa,\|\beta^h\|_{C^2(\partial\Omega)},\|g\|_{W^{1,\infty}(\Omega)})>0$ such that the following hold.

\begin{enumerate}[label=(\roman*)]
\item The function $y^{\varepsilon}$ satisfies
\begin{align*}
L[y^{\varepsilon}] + B^{\varepsilon}\cdot\nabla y^{\varepsilon}\leq C.
\end{align*}
Here, $B$ is the vector field given by
\[
(B^{\varepsilon})^j = -h \partial_i (\phi^{\varepsilon}_{ij}(\nabla w^{\varepsilon}))\qquad\text{for each }j.
\]

\item The function $z^{\varepsilon}=\rho y^{\varepsilon}$ satisfies
\begin{align*}
L[z^{\varepsilon}] + B^{\varepsilon}\cdot\nabla z^{\varepsilon}\leq C.
\end{align*}
\end{enumerate}


\end{proposition}



\begin{proof}
In this proof, we drop the superscript $\varepsilon$ and adopt the Einstein convention for brevity.

\medskip

(i) \textbf{Step 1. Linearize the equation.} Consider the regularized equation for $w$:
\begin{equation}\label{eq:main}
w - h \partial_i (\phi_i(\nabla w)) = g 
\end{equation}
where $\phi(p) = \sqrt{\varepsilon^2 + |p|^2}$, $\phi_i = \frac{p_i}{\phi}$, and $\phi_{ij} = \frac{1}{\phi} \left( \delta_{ij} - \frac{p_i p_j}{\phi^2} \right)$. Define the linearized operator $L$ and the vector field $B$ as
\[
L[u] = u - h \phi_{ij}(\nabla w) \partial_{ij} u, \quad B^j = -h \partial_i (\phi_{ij}(\nabla w)).
\]
Differentiating \eqref{eq:main} with respect to $x_k$ and using the chain rule ($\partial_k \phi_i = \phi_{ij} w_{jk}$), we derive
\[
w_k - h \partial_i (\partial_k \phi_i) = g_k \quad\text{and thus}\quad w_k - h \phi_{ij} w_{ijk} - h (\partial_i \phi_{ij}) w_{jk} = g_k.
\]
Rearranging to solve for the third-order term yields the identity
\begin{equation}\label{eq:third_deriv}
h \phi_{ij} w_{ijk} = w_k - g_k + B^j w_{jk}. 
\end{equation}

\medskip

\textbf{Step 2. Differentiate the function $y$.} Let $A^k = \beta^h \partial_k \eta$ and $\psi = A^k w_k$. We compute the derivative of $y = \phi + \psi$ to have
\begin{equation}\label{eq:first_y}
\partial_i y = \frac{w_k w_{ki}}{\phi} + A^k w_{ki} + (\partial_i A^k) w_k.
\end{equation}
The second-order derivatives of $\phi$ and $\psi$ are
\[
\partial_{ij} \phi = \frac{w_k w_{kij}}{\phi} + \frac{1}{\phi} (w_{ki} w_{kj} - \partial_i \phi \partial_j \phi)
\]
and
\[
\partial_{ij} \psi = A^k w_{kij} + (\partial_i A^k) w_{kj} + (\partial_j A^k) w_{ki} + (\partial_{ij} A^k) w_k.
\]

\medskip

\textbf{Step 3. Apply the operator $L$ to $y$.} By linearity, $L[y] = L[\phi] + L[\psi]$. Grouping the terms containing third derivatives $w_{kij}$, we have
\begin{align*}
L[y] &= (\phi + A^k w_k) - h \phi_{ij} \left( \frac{w_k}{\phi} + A^k \right) w_{kij} \\
&\quad - \frac{h}{\phi} \phi_{ij} (w_{ki} w_{kj} - \partial_i \phi \partial_j \phi) - h \phi_{ij} \left( 2 (\partial_i A^k) w_{kj} + (\partial_{ij} A^k) w_k \right).
\end{align*}
We substitute the identity for $h \phi_{ij} w_{ijk}$ from equation \eqref{eq:third_deriv} to derive
\begin{align*}
L[y] = y - \left( \frac{w_k}{\phi} + A^k \right) (w_k - g_k + B^j w_{jk})  + Q - h \phi_{ij} \left( 2 (\partial_i A^k) w_{kj} + (\partial_{ij} A^k) w_k \right).
\end{align*}
Here, we let the quadratic term $Q := - \frac{h}{\phi} \phi_{ij} (w_{ki} w_{kj} - \partial_i \phi \partial_j \phi)$, which will be analyzed in a moment.

Focusing on the first two terms, we expand the product and use $y - A^k w_k = \phi$ and $\phi - \frac{|\nabla w|^2}{\phi} = \frac{\varepsilon^2}{\phi}$ to have
\[
L[y] = \frac{\varepsilon^2}{\phi} + \left( \frac{w_k}{\phi} + A^k \right) g_k - \left( \frac{w_k w_{jk}}{\phi} + A^k w_{jk} \right) B^j + Q - h \phi_{ij} \left( 2 (\partial_i A^k) w_{kj} + (\partial_{ij} A^k) w_k \right).
\]
From \eqref{eq:first_y}, we have
\[
\frac{w_k w_{jk}}{\phi} + A^k w_{jk} = \partial_j y - (\partial_j A^k) w_k,
\]
and thus, we have
\begin{align}
L[y]+B\cdot\nabla y
&=R_0+Q+B^j(\partial_jA^k)w_k
-2h\phi_{ij}(\partial_iA^k)w_{kj},
\label{eq:L-before-Q-corrected}\\
\text{where } R_0
&:=\frac{\varepsilon^2}{\phi}
+\left(\frac{w_k}{\phi}+A^k\right)g_k
-h\phi_{ij}(\partial_{ij}A^k)w_k.
\label{eq:R0}
\end{align}

\medskip

\textbf{Step 4. Estimate the terms concerning the Hessian $w_{ij}$}. The goal is to estimate the terms
\[
Q  + B^j (\partial_j A^k) w_k - 2h\phi_{ij}(\partial_iA^k)w_{kj}.
\]
We set
\[
q_k:=\frac{w_k}{\phi},
\qquad
P_{k\ell}:=\delta_{k\ell}-q_kq_\ell,
\qquad
\Phi:=(\phi_{ij}),
\qquad
H:=(w_{ij}).
\]
Then
\begin{equation}\label{eq:P-Phi}
P=\phi\Phi,
\qquad |q|\leq1,
\qquad 0<\Phi\leq\frac{1}{\phi}I.
\end{equation}
We further set
\[
S:=\Phi^{1/2}H\Phi^{1/2},
\qquad
X:=\Phi H
\]
and use the notation $M:N=\operatorname{tr}(M^{\mathsf T}N)$ and $|M|^2=M:M$ henceforth.

By direct differentiation and using $\phi_{kl}=\frac{1}{\phi}\left(\delta_{kl} - \frac{w_kw_l}{\phi^2}\right)$, we compute the quadratic term $Q$ by
\begin{align*}
Q &= - \frac{h}{\phi} \phi_{ij} (w_{ki} w_{kj} - \partial_i \phi \partial_j \phi)  = - \frac{h}{\phi} \phi_{ij} \left(w_{ki} w_{lj}\delta_{kl} - \frac{w_kw_{ki}w_lw_{lj}}{\phi^2}\right) \notag \\
&= - \frac{h}{\phi} \phi_{ij} w_{ki} w_{lj}\left(\delta_{kl} - \frac{w_kw_l}{\phi^2}\right) = -h\phi_{ij} w_{ki} w_{lj}\phi_{kl}. \notag
\end{align*}
Therefore, we obtain
\begin{align}\label{eq:true-Q}
Q = -h|S|^2.
\end{align}



Set $G=(G_{jk})$ where
\[
G_{jk} := \partial_j A^k\qquad\text{for each }j,k.
\]
We now expand the other terms by writing
\[
B^j (\partial_j A^k) w_k = -h \partial_i (\phi_{ij}(\nabla w))(\partial_j A^k) w_k = -h\phi_{ijm}w_{mi}(\partial_j A^k)w_k.
\]
By direct differentiation, it holds that
\[
\phi_{ijm} = \frac{3w_i w_j w_m}{\phi^5} - \frac{\delta_{ij} w_m + \delta_{im} w_j + \delta_{jm} w_i}{\phi^3} = -\frac{\phi_{ij}w_m+\phi_{im}w_j+\phi_{jm}w_i}{\phi^2}.
\]
Therefore, after appropriate re-indexing and summing over all $j,k$, we can write,
\begin{align*}
B^j(\partial_jA^k)w_k
&=2hX_{jk}(G_{jm}q_m)q_k
+h(q_\ell G_{\ell m}q_m) X_{jk}\delta_{jk}.
\end{align*}
It follows that
\begin{equation}\label{eq:other-terms-corrected}
B^j(\partial_jA^k)w_k
-2h\phi_{ij}(\partial_iA^k)w_{kj}
=-2hX:U,
\end{equation}
where
\begin{equation*}
U = (U_{jk}),\qquad U_{jk}
:=G_{jk}-(G_{jm}q_m)q_k
-\frac12(q_\ell G_{\ell m}q_m)\delta_{jk}.
\end{equation*}
Equivalently,
\[
U=GP-\frac12(q^{\mathsf T}Gq)I.
\]
Define
\begin{equation*}
\widehat U:=\Phi^{1/2}U\Phi^{-1/2}.
\end{equation*}
Using $P=\phi\Phi$, we obtain
\[
\widehat U
=\phi\Phi^{1/2}G\Phi^{1/2}
-\frac12(q^{\mathsf T}Gq)I.
\]
Using the fact that $|\phi\Phi^{1/2}G\Phi^{1/2}| \leq \phi\|\Phi\|_2|G|$ where $\|\Phi\|_2\left(=\frac{1}{\phi}\right)$ is the maximum eigenvalue of $\Phi$ (and therefore $\|\Phi\|_2=\frac{1}{\phi}$), we have
\begin{equation}\label{eq:Uhat-bound}
|\widehat U|\leq C|G|\leq C,
\end{equation}
where $C>0$ is independent of $\varepsilon$ and $h$.
Moreover,
\[
X:U=S:\widehat U.
\]
Combining \eqref{eq:true-Q} and \eqref{eq:other-terms-corrected}, we derive
\begin{align}
&Q+B^j(\partial_jA^k)w_k
-2h\phi_{ij}(\partial_iA^k)w_{kj}\notag\\
&\qquad=-h|S|^2-2hS:\widehat U
=-h|S+\widehat U|^2+h|\widehat U|^2
\leq Ch.
\label{eq:Hessian-i}
\end{align}

\medskip 

\textbf{Step 5. Conclude.} Therefore, by \eqref{eq:L-before-Q-corrected} and \eqref{eq:Hessian-i}, and by the fact that the other terms of $R_0$ are bounded by some constant (that is, $\frac{\varepsilon^2}{\phi}\leq \varepsilon$ and $\phi_{ij}(\partial_{ij}A^k)w_k \leq C$ in \eqref{eq:R0}), we conclude
\[
L[y] + B \cdot \nabla y \le C.
\]

\medskip

(ii) \textbf{Step 1. Apply the operator $L$ to $z = \rho y$.} Using the product rule for second derivatives, we have
\[ \partial_{ij}(\rho y) = \rho_{ij} y + \rho_i y_j + \rho_j y_i + \rho y_{ij}. \]
Substituting this into the operator $L$ yields
\begin{align*}
L[z] &= \rho y - h \phi_{ij} (\rho_{ij} y + 2 \rho_i y_j + \rho y_{ij}) \\
&= \rho (y - h \phi_{ij} y_{ij}) - h y (\phi_{ij} \rho_{ij}) - 2h \phi_{ij} \rho_i y_j \\
&= \rho L[y] - h y (\phi_{ij} \rho_{ij}) - 2h \phi_{ij} \rho_i y_j.
\end{align*}
Thus, we have
\begin{equation*}
L[z] = \rho \left(L[y] + B \cdot \nabla y \right) - \rho B \cdot \nabla y - h y (\phi_{ij} \rho_{ij}) - 2h \phi_{ij} \rho_i y_j. 
\end{equation*}

From $z = \rho y$, we change the derivatives of $y$ to those of $z$ by using
\[ \nabla z = \rho \nabla y + y \nabla \rho \implies \rho \nabla y = \nabla z - y \nabla \rho. \]
so that we obtain
\begin{align}\label{eq:Lz-corrected}
L[z] + B\cdot \nabla z= \rho \left(L[y] + B \cdot \nabla y \right) + y B^j\rho_j - h y (\phi_{ij} \rho_{ij}) - 2h \phi_{ij} \rho_i y_j.
\end{align}

\medskip

\textbf{Step 2. Estimate the terms linear in the Hessian $X=D^2w$ as before.}
Write
\[
c:=\frac{y}{\phi}=1+A^kq_k.
\]
Since $|A|\leq1-2\kappa$ and $|q|\leq1$, it holds that
\begin{equation}\label{eq:c-bound}
2\kappa\leq c\leq2-2\kappa.
\end{equation}
From \eqref{eq:first_y},
\begin{equation}\label{eq:y-gradient-matrix}
y_j=(q_k+A^k)w_{kj}+G_{jk}w_k.
\end{equation}
We record the equation of $z$ after substituting \eqref{eq:y-gradient-matrix}:
\begin{align}\label{eq:Lz-corrected-1}
L[z] + B\cdot \nabla z= \rho \left(L[y] + B \cdot \nabla y \right) + y B^j\rho_j - 2h \phi_{ij} \rho_i (q_k + A^k)w_{kj} - 2h \phi_{ij} \rho_i G_{jk}w_k - h y (\phi_{ij} \rho_{ij}) .
\end{align}

On the other hand, recall the definition $X=\Phi H$ and the formula
\[
\phi_{ijk} = -\frac{\phi_{ij}w_k+\phi_{ik}w_j+\phi_{jk}w_i}{\phi^2}.
\]
As before, we obtain
\begin{equation}\label{eq:B-rho}
B^j\rho_j = -h\phi_{ijk}w_{ki}\rho_j = \frac{h}{\phi}
\left(2\rho_jX_{jk}q_k+(q_j\rho_j)\operatorname{tr}X\right).
\end{equation}
Combining \eqref{eq:y-gradient-matrix} and \eqref{eq:B-rho}, the terms $y B^j\rho_j - 2h\phi_{ij}\rho_i(q_k+A^k)w_{kj}$ in
\eqref{eq:Lz-corrected} after being obtained by substituting \eqref{eq:y-gradient-matrix}, which are linear in $H$, satisfy, with suitable re-indexing,
\begin{align}
&yB^j\rho_j
-2h\phi_{ij}\rho_i(q_k+A^k)w_{kj} \notag\\
&\quad=h\left(
2c\rho_jX_{jk}q_k
+c(q_j\rho_j)\operatorname{tr}X
-2\rho_jX_{jk}(q_k+A^k)
\right) \notag\\
&\quad=-2hX:V,\label{eq:into-V}
\end{align}
where
\begin{align*}
V=(V_{jk}),\qquad V_{jk}
:= - c\rho_jq_k-\frac12c\rho_jq_j\delta_{jk}+\rho_j(q_k+A^k).
\end{align*}
By using $c-1 = A^mq_m$ and $(PA)_k = A^k - (A\cdot q)q_k$, we can simplify $V_{jk}$ by writing
\begin{align*}
V_{jk} = \rho_j(A^k + (1-c)q_k) -\frac12c\rho_jq_j\delta_{jk} & = \rho_j(A^k - A^mq_m q_k) -\frac12c\rho_jq_j\delta_{jk} \\
& = \rho_j(PA)_k -\frac12c\rho_jq_j\delta_{jk}.
\end{align*}
Define
\begin{equation*}
\widehat V:=\Phi^{1/2}V\Phi^{-1/2}.
\end{equation*}
Since $PA=\phi\Phi A$, we have
\[
\widehat V
=(\Phi^{1/2}\nabla\rho)
\otimes(\phi\Phi^{1/2}A)
-\frac{c}{2}(q\cdot\nabla\rho)I.
\]
Similarly as before, we have, from \eqref{eq:P-Phi}, \eqref{eq:c-bound}, and the boundedness
of $\rho$ and its derivatives, that
\begin{equation}\label{eq:Vhat-bound}
|\widehat V|\leq C,
\end{equation}
uniformly in $\varepsilon$ and $h$, and that
\[
X:V=S:\widehat V.
\]

\medskip

\textbf{Step 3. Estimate the terms concerning the Hessian $w_{ij}$ and conclude.}
Substituting the expansion \eqref{eq:L-before-Q-corrected} in
\eqref{eq:Lz-corrected-1} and using \eqref{eq:Hessian-i}, \eqref{eq:into-V}, we obtain
\begin{align}
L[z]+B\cdot\nabla z
={}&\rho R_0-hy\phi_{ij}\rho_{ij}
-2h\phi_{ij}\rho_iG_{jk}w_k\notag\\
&\quad-\rho h|S|^2
-2hS:\bigl(\rho\widehat U+\widehat V\bigr).
\label{eq:Lz-final-decomposition}
\end{align}
Let
\[
\widehat W:=\rho\widehat U+\widehat V.
\]
By \eqref{eq:Uhat-bound}, \eqref{eq:Vhat-bound}, and $\rho\geq1$, we have
\begin{align}
-\rho h|S|^2-2hS:\widehat W
&=-\rho h\left|S+\frac{\widehat W}{\rho}\right|^2
+\frac{h}{\rho}|\widehat W|^2
\leq Ch.
\label{eq:Hessian-ii}
\end{align}
Finally,
\[
|y|\,|\Phi|\leq C,
\qquad
|\Phi|\,|\nabla w|\leq C,
\qquad
|A|+|DA|+|D^2A|+|\rho|+|D\rho|+|D^2\rho|\leq C.
\]
Hence all the terms in the first line on the right-hand side of
\eqref{eq:Lz-final-decomposition} are bounded uniformly in $\varepsilon$
and $h$. Combining this fact with \eqref{eq:Hessian-ii}, we conclude
\[
L[z]+B\cdot\nabla z\leq C.
\]
This proves (ii) and finishes the proof.
\end{proof}

We are now in the position to prove the Lipschitz bound estimate.
\begin{proof}[Proof of Proposition \ref{prop:a-priori-estimate}]

Thanks to the maximum principle (see \cite{PW84}) applied to the following boundary value problem (as a consequence of Propositions \ref{prop:multiplier} and \ref{prop:equation-v}(ii))
\begin{equation*}
\begin{cases} 
\displaystyle L[z^{\varepsilon}] + B^{\varepsilon}\cdot\nabla z^{\varepsilon}\leq C & \text{in} \quad \Omega, \\ 
\xi^{\varepsilon}(z^{\varepsilon}) \leq0 & \text{on} \quad \partial\Omega,
\end{cases}
\end{equation*}
we obtain that $$z^{\varepsilon}\leq C\qquad\text{on}\quad\overline{\Omega}.$$
Since $z^{\varepsilon} = \rho y^{\varepsilon}$ and $2\kappa\phi^{\varepsilon} \leq y^{\varepsilon} \leq (2-2\kappa)\phi^{\varepsilon}$, we derive that $$\phi^{\varepsilon}\leq C,\quad\text{which leads to}\quad |\nabla w^{\varepsilon}|\leq C\quad\text{on}\quad\overline{\Omega}$$
for some constant $C=C(\Omega,\kappa,\|\beta^h\|_{C^2(\partial\Omega)},\|g\|_{W^{1,\infty}(\Omega)})>0$. This yields the conclusion of the proposition and finishes the proof.

Finally, we note that $w^\varepsilon$ is uniformly bounded and equi-continuous with respect to $\varepsilon$.
Hence, we deduce from the Arzer\'a--Ascoli theorem that $w^\varepsilon$ locally uniformly converges to some function $w$ as $\varepsilon\to 0$, which also ensures the strong convergence of $w^\varepsilon$ to $w$ in $L^2(\Omega)$.
Therefore, we can invoke Proposition~\ref{prop:gamma-convergence} to conclude that the limit function $w$ corresponds to the unique solution of the minimization problem \eqref{eq:minimization-with-g}. 
\end{proof}

\section{Convergence result}
\label{sec:convergence}
In this section, we shall give a proof of Theorem~\ref{thm:main}.
The proof consists in verifying that the function operator $S_h$ defined in \eqref{eq:def-Sh} satisfies the criteria that have been presented in Proposition~\ref{prop:BS}.
To this end, we prepare several lemmas, although some parts of the proofs will be left to previous works \cite{EG24,EG25} in the case when the proof might become lengthy.
From now on, we assume that the schemes $T_h$ and $S_h$ are constructed with the energy $J_{\beta^h}(u)$, where smooth functions $\beta^h$ on $\partial\Omega$ satisfying $\|\beta^h\|_{C^0(\partial\Omega)} \leq \|\beta\|_{C^0(\partial\Omega)}$ uniformly converge to $-\beta$ on $\pOmega$ as $h\to 0$.

\begin{lemma}[Monotonicity of $T_h$]
    \label{lem:Th_monotone}
    The set operator $T_h$ is monotone; for every $E\subset F\subset\closure{\Omega}$, it holds that $T_h(E)\subset T_h(F)$.
\end{lemma}
\begin{proof}
    Let $w_E$ be the unique solution of the variational problem \eqref{eq:mms}.
    Then, we have $\{w_E\leq 0\}$. Since $E\subset F$, it is immediate that $d_{\Omega,E}\geq d_{\Omega,F}$ in $\closure{\Omega}$.
    We deduce from \cite[Lemma 3.1]{EG25} that $w_E\geq w_F$ in $\closure{\Omega}$ which concludes the proof.
\end{proof}

\begin{lemma}[Monotonicity of $S_h$]
    \label{lem:Sh_monotone}
    The function operator $S_h$ is monotone; for every $u,v\in C(\closure{\Omega})$ satisfying $u\leq v$ in $\closure{\Omega}$,
    it holds that $$S_hu\leq S_hv\qquad\mbox{in}\quad\closure{\Omega}.$$
\end{lemma}
\begin{proof}
    Take any $u,v\in C(\closure{\Omega})$ with $u\leq v$ in $\closure{\Omega}$, and $x\in\closure{\Omega}$.
    Then, for every $\lambda\in\R$, it immediately follows that $\{u\geq \lambda\} \subset \{v\geq \lambda\}$.
    Since $T_h$ is monotone by Lemma~\ref{lem:Th_monotone}, we obtain that $T_h(\{u\geq \lambda\})\subset T_h(\{v\geq \lambda\})$,
    which concludes the proof.
\end{proof}

\begin{lemma}[Translation invariance of $S_h$]
    \label{lem:Sh_translation}
    The function operator $S_h$ is translation invariant; for every $u\in C(\closure{\Omega})$ and $c\in\R$, it holds that $$S_h(u + c) = S_hu + c\qquad\mbox{in}\quad\closure{\Omega}.$$
\end{lemma}
\begin{proof}
    This is straightforward from the definition of $S_h$.
\end{proof}

\begin{lemma}[Continuity of $T_h$]
    \label{lem:Th_continuity} The set operator $T_h$ is continuous in the sense of Definition~\ref{defn:Th-continuous}.
\end{lemma}
\begin{proof}
    Since $T_h$ is defined by $E\mapsto T_h(E):=\{w_E\leq 0\}$, it is enough to confirm that for every $E\subset\closure{\Omega}$, $w_E$ satisfies
    all conditions in Proposition~\ref{prop:continuity-T}.
    The monotonicity has been shown in \cite[Lemma 3.1]{EG25}.
    Since $\Omega$ is bounded, there exists a constant $M > 0$ satisfying $-M\leq d_{\Omega,E}\leq M$ for all $E\subset\closure{\Omega}$.
    Since the constant functions $w_\pm \equiv\pm M$ are classical solutions to the discrete equation \eqref{eq:E-L} with $g=\pm M$, the monotonicity property again yields $-M\leq w_E\leq M$, and thus $w_E$ is uniformly bounded with respect to $E$.

    Finally, we shall check the equi-continuity of $w_E$ with respect to $E$. We deduce from Proposition~\ref{prop:a-priori-estimate} that there exists a positive constant $C$ depending only on $\Omega$, $\kappa$, and $\|\sgrad{\beta}\|_{C^1(\pOmega)}$ (and the constant $C_\Omega$ from Proposition~\ref{prop:geodesic-uniform-bounded} which depends only on $\Omega$) such that  $$\|\nabla w_E\|_{L^\infty(\Omega)}\leq C.$$
    Since the right-hand side of this inequality is independent of $E$, we see that $w_E$ is uniformly Lipschitz continuous with respect to $E$.
\end{proof}


\begin{lemma}[Consistency of $S_h$]
    \label{lem:Sh_consistency-main}
    For every $\varphi\in C^2_F(\Omega)\cap C^2(\overline{\Omega})$ and $z\in\overline{\Omega}$ satisfying either
    \begin{itemize}
        \item $z\in\Omega$ or
        \item $z\in\partial\Omega$ and $\left<\nabla\varphi(z),\nu_\Omega(z)\right> + \beta(z)|\nabla\varphi(z)| > 0\, (resp., < 0)$,
    \end{itemize}
    it holds that
    \begin{equation}
        \label{eq:consistency-main}
        \begin{aligned}
        \relaxlimitSup\frac{S_h\varphi(z) - \varphi(z)}{h}  &\le - F_*(\nabla\varphi(z), \nabla^2\varphi(z))\\
        resp.,\, \relaxlimitInf\frac{S_h\varphi(z) - \varphi(z)}{h}  &\ge - F^*(\nabla\varphi(z), \nabla^2\varphi(z)).
        \end{aligned}
    \end{equation}
\end{lemma}
\begin{proof}[Sketch of proof]
We follow the consistency argument in \cite[Theorem 4.3]{EG25}.
We recall the points needed for the present paper.

First, by Lemma~\ref{lem:Th_continuity} and
Proposition~\ref{prop:relation-TS}, we have
\[
    \{S_h\varphi\ge\lambda\}
    =
    T_h(\{\varphi\ge\lambda\})\qquad\forall\lambda\in\R.
\]
Thus, the desired consistency inequalities are reduced to one-step
inclusion estimates for
\[
    E^\varphi_\lambda:=\{\varphi\ge\lambda\}.
\]

Second, near a regular level set of \(\varphi\), the argument of
\cite[Proposition 4.2]{EG25} gives the one-sided local expansion
\[
    w^h_{E^\varphi_\lambda}
    =
    d_{\Omega,E^\varphi_\lambda}
    -
    h\kappa_{E^\varphi_\lambda}
    +
    o(h)\qquad\mbox{as}\quad h\to 0,
\]
where $\kappa_{E^\varphi_\lambda}$ denotes the mean curvature of the smooth hypersurface $\partial E^\varphi_\lambda$, say $\kappa_{E^\varphi_\lambda} := \operatorname{div}(\nabla\phi(\nabla\varphi))$,
in the sense needed for the comparison argument. At boundary points on $\pOmega$,
the sign condition
\[
    \langle\nabla\varphi(z),\nu_\Omega(z)\rangle
    +\beta(z)|\nabla\varphi(z)|\gtrless 0
\]
allows one to choose $h>0$ so small that
\[
    \langle\nabla\varphi(z),\nu_\Omega(z)\rangle
    - \beta^h(z)|\nabla\varphi(z)|\gtrless 0,
\]
and $\underline{\beta}\leq \beta^h\leq \overline{\beta}$ holds on $\pOmega$ for some $-1 < \underline{\beta}\leq \overline{\beta}<1$.
Thus, we can use the same translating-soliton barriers as in
\cite{EG25} with $\beta = \beta^h$. As in that proof, the test function may be modified outside
a small neighborhood of \(z\), and the monotonicity of \(T_h\) is used to
compare the corresponding super-level sets.

Third, inserting this one-step estimate into the relation between
\(S_h\) and \(T_h\), and expanding \(\varphi\) at the points
\(z_h\to z\), gives
\[
    \relaxlimitSup
    \frac{S_h\varphi(z)-\varphi(z)}{h}
    \le
    -F_*(\nabla\varphi(z),\nabla^2\varphi(z)),
\]
and similarly for the lower inequality. This is precisely the last part
of the proof of \cite[Theorem 4.3]{EG25}.

Finally, if \(\nabla\varphi(z)=0\) and \(\nabla^2\varphi(z)=O\), the flat
barrier argument of \cite[Theorem 4.3]{EG25} applies and yields both
upper and lower limits, since \(F_*(0,O)=F^*(0,O)=0\).
\end{proof}

We are now in the position to prove our main result.

\begin{proof}[\textbf{Proof of Theorem~\ref{thm:main}}]

    We follow the strategy of Barles and Souganidis \cite[Theorem 2.1]{BS91}.
    Since $\overline{u}\geq \underline{u}$ is trivial, we aim to show $\overline{u}\leq \underline{u}$ in $\closure{\Omega}\times[0,T]$.
    By the definition, it immediately follows that $\overline{u}(\cdot,0) = u_0 = \underline{u}(\cdot,0)$ in $\closure{\Omega}$.
    Thus, the comparison principle (Proposition~\ref{prop:comparison-principle1}) implies $\overline{u}\leq \underline{u}$ in $\closure{\Omega}\times[0,T]$ once
    $\overline{u}$ (resp., $\underline{u}$) is shown to be a viscosity subsolution (resp., supersolution).

    Hereafter, we only show the former case since the latter one can be proved by a symmetric argument.
    Take any $\varphi\in C^2_F(\closure{\Omega}\times[0,T])$ and $(\hat x, \hat t)\in\closure{\Omega}\times[0,T]$.
    Assume that $\overline{u} - \varphi$ takes a local maximum at $(\hat x,\hat t)$.
    We need to show
    \begin{equation}
        \label{eq:convergence-goal}
        \varphi_t(\hat x, \hat t) + F_*(\nabla\varphi(\hat x, \hat t), \nabla^2\varphi(\hat x,\hat t)) \leq 0.
    \end{equation}
    Then, there exist sequences $\{h_n\}_n$ and $\{(x_n,t_n)\}_n$ such that
    \begin{equation*}
        h_n \downarrow 0,\quad (x_n,t_n)\to (\hat x,\hat t),\quad
        u^{h_n}(x_n,t_n)\to \overline{u}(\hat x, \hat t)\qquad\mbox{as}\quad n\to\infty
    \end{equation*}
    and $u^{h_n} - \varphi$ takes the global maximum at $(x_n,t_n)$.
    This is possible since $u^h$ is uniformly bounded with respect to $h > 0$, and we can modify $\varphi$ outside a ball $B_\delta(\hat x,\hat t)$.

    First, we treat the case when $\nabla\varphi(\hat x,\hat t)\neq 0$, and assume that
    either $\hat x\in\Omega$ or $\hat x\in\pOmega$ and $\langle\nu_\Omega(\hat x),\nabla\varphi(\hat x,\hat t)\rangle + \beta(\hat x)|\nabla \varphi(\hat x,\hat t)| > 0$ (otherwise, the conclusion is straightforward due to the definition of viscosity subsolutions).
    Then, we have

    \begin{equation}
        \label{eq:convergence-1}
    u^{h_n}(\cdot, t_n - h_n) - \varphi(\cdot, t_n-h_n) \leq u^{h_n}(x_n, t_n) - \varphi(x_n,t_n)\qquad\mbox{in}\quad \closure{\Omega}.
    \end{equation}
    Let $\xi_n$ denote the right-hand side of the above inequality.
    We apply $S_{h_n}$ to both sides of \eqref{eq:convergence-1} and obtain that

    \begin{equation}
        \label{eq:convergence-2}
        u^{h_n}(\cdot, t_n) \leq S_{h_n}\left[\varphi(\cdot, t_n - h_n)\right] + \xi_n\qquad\mbox{in}\quad \closure{\Omega}.
    \end{equation}
    Here, we have invoked the monotonicity and translation invariance of $S_{h_n}$ (Lemma~\ref{lem:Sh_monotone} and Lemma~\ref{lem:Sh_translation} on noting that $\xi_n$ is constant).
    Subtracting $\varphi(x_n,t_n-h_n)$ from both sides of \eqref{eq:convergence-2} and evaluating at $x = x_n$,
    we derive

    \begin{equation*}
        \varphi(x_n, t_n) - \varphi(x_n, t_n-h_n) \leq S_{h_n}\left[\varphi(\cdot,t_n-h_n)\right](x_n) - \varphi(x_n,t_n-h_n).
    \end{equation*}
    We divide the above inequality by $h_n$ and send $n\to\infty$ to obtain
    \begin{align*}
    \varphi_t(\hat x,\hat t) &= \lim_{n\to\infty}\frac{\varphi(x_n,t_n) - \varphi(x_n,t_n-h_n)}{h_n}\\
    &\leq \relaxlimitSup\left[\frac{S_h\varphi(\cdot,\hat t) - \varphi(\cdot,\hat t)}{h}\right](\hat x)\leq -F_*(\nabla\varphi(\hat x,\hat t), \nabla^2\varphi(\hat x,\hat t)).
    \end{align*}
    Here, we have invoked the consistency property of $S_{h_n}$ (Lemma~\ref{lem:Sh_consistency-main}) to derive the last inequality,
    and hence \eqref{eq:convergence-goal} follows.

    In the case when $\nabla\varphi(\hat x,\hat t) = 0$, we may assume that $\nabla^2\varphi(\hat x,\hat t) = O$ (see \eqref{eq:compatible-set}).
    Since $F_*(0,O) = 0$ holds for our definition of $F$, the previous argument works, and we see that $\varphi_t(\hat x,\hat t) \leq 0$.

    Therefore, we conclude that $\overline{u}$ is a viscosity subsolution of \eqref{eq:level-set-FB}.
\end{proof}

\begin{remark}
We note that the limit of the modified Capillary Chambolle-type scheme is independent of the choice of a convergent sequence $\beta^h \to -\beta$ in $C^0(\pOmega)$.
\end{remark}


\section*{Data availability} Data sharing is not applicable to this article (no datasets were generated or analyzed during the study).

\section*{Declaration} The authors declare that they have no relevant financial or non-financial interests related to this work to disclose.

\section*{Acknowledgment}

The work of T. E. is supported in part by the ANR (Agence Nationale de la Recherche) and grant ANR-23-PEIA-0004 (PEPR IA, PDE-AI). The work of J. J. is supported in part by the National Science Foundation under Grant No.~2529637. Any opinions, findings, and conclusions or recommendations expressed in this material are those of the authors and do not necessarily reflect the views of the National Science Foundation.

The authors appreciate Mr. Kei Matsushita at the University of Tokyo for his question which motivated further improvement of Section 3.

\bibliographystyle{abbrv}
\bibliography{Preprint}

\end{document}